\documentclass[english]{lmcs}
\pdfoutput=1

\usepackage{lastpage}
\lmcsdoi{21}{3}{29}
\lmcsheading{}{\pageref{LastPage}}{}{}%
{Dec.~11,~2023}{Sep.~24,~2025}{}

\usepackage{comment}
\usepackage[T1]{fontenc}
\usepackage[utf8]{inputenc}
\usepackage{babel}
\usepackage{amstext}
\usepackage{amssymb}

\usepackage[unicode=true,pdfusetitle,
 bookmarks=true,bookmarksnumbered=false,bookmarksopen=false,
 breaklinks=false,pdfborder={0 0 0},pdfborderstyle={},backref=false,colorlinks=true]
 {hyperref}
\hypersetup{
 linkcolor=brown, citecolor=blue, pdfstartview={FitH}, hyperfootnotes=false, unicode=true}

\makeatletter
\numberwithin{equation}{section}
\numberwithin{figure}{section}
\theoremstyle{plain}

\usepackage{tikz-cd}
\usepackage{bbm}

\makeatother
\begin{document}
\title[Multi-representation associated to a numbered subbasis ]{Multi-representation associated to the numbering of a subbasis and
formal inclusion relations}
\author[E. Rauzy]{Emmanuel Rauzy \lmcsorcid{0009-0005-0395-3311}}

\address{Universität der Bundeswehr München, München, Germany}

\thanks{The author is funded by an Alexander von Humboldt Research Fellowship.}
\begin{abstract}
We revisit Dieter Spreen's notion of a representation associated to
a numbered basis equipped with a strong inclusion relation. We show
that by relaxing his requirements, we obtain different classically
considered representations as subcases, including representations
considered by Grubba, Weihrauch and Schröder. We show that the use
of an appropriate strong inclusion relation guarantees that the representation
associated to a computable metric space seen as a topological space
always coincides with the Cauchy representation. We also show how
the use of a formal inclusion relation guarantees that when defining
multi-representations on a set and on one of its subsets, the obtained
multi-representations will be compatible, i.e. inclusion will be a
computable map. The proposed definitions are also more robust under
change of equivalent bases. 
\end{abstract}

\maketitle

\section{Introduction}

In order to study computability in areas of mathematics where mathematician
freely define very abstract objects, one has first to answer the question:
how can a machine manipulate an abstract object? 

Turing, in his seminal paper \cite{Turing1937}, gave a first approach
to this problem, and he was able to define computable functions of
a computable real variable. 

Kleene's general solution is known as \emph{realizability}: to represent
abstract objects by concrete descriptions, thanks to the use of semantic
functions which give meaning to a priori inert symbols. A function
between abstract objects is then called \emph{computable} if it can
be realized by a computable function on concrete objects. 

In the Type Two theory of Effectivity (TTE), the considered set of
concrete objects is the Baire space $\mathbb{N}^{\mathbb{N}}$, and
the notion of computability is given by Turing machines that work
with infinite tapes. The semantic functions that map elements of the
Baire space to abstract objects are called \emph{representations},
they were introduced by Kreitz and Weihrauch in \cite{Kreitz1985}.
This notion was extended by Schröder to \emph{multi-representations}
in his dissertation \cite{Schroeder2003}.

\medskip

One of the defining features of TTE is that the study of computable
functions is always related to the study of continuity, because, on
the Baire space, a function is continuous if and only if it is computable
with respect to some oracle. Thanks to Schröder's generalization of
Weihrauch's notion of \emph{admissible representation} \cite{Schroeder2002},
this phenomenon extends to other topological spaces. One can then
investigate in parallel computability and continuity, and reductions
between problems, or translations between representations, exist both
in terms of computable functions and in terms of continuous functions. 

\medskip

A celebrated theorem of Schröder \cite{Schroeder2003} characterizes
those topological spaces that admit admissible multi-representations
as those whose sequentialization is the quotient of countably based spaces (qcb-spaces). On such a space,
there is a unique admissible multi-representation, up to continuous
translation. 

However, when studying computability, representations are considered
up to computable translation. 

One could of course ask to distinguish amongst admissible representations
those that are appropriate to study computability -and possibly call
them ``computably admissible''. 

However, because qcb-spaces can have continuously many auto-homeomorphisms,
there is no hope of distinguishing a single equivalence class of representation
as \emph{the} correct one to study computability. Indeed, if $\rho:\subseteq\mathbb{N}^{\mathbb{N}}\rightarrow X$
is an admissible representation of a topological space $X$, which
we know to be appropriate for studying computability, then for any
auto-homeomorphism $\Theta:X\rightarrow X$ of $X$, $\Theta\circ\rho$
will be another representation of $X$ which will have exactly the
same properties as $\rho$. The representation $\Theta\circ\rho$
is computably equivalent to $\rho$ exactly when $\Theta$ and its
inverse are $(\rho,\rho)$-computable, and thus there can only be
countably many of these representations that are computably equivalent. 

In practice, there is often a single correct choice of a representation
on a set, but this comes from the fact that we consider sets that
have more intrinsic structure than just a topology. For example, any
permutation of $\mathbb{N}$ is an auto-homeomorphism of $\mathbb{N}$
for the discrete topology, and thus there is no hope of fixing \emph{the}
correct representation of $\mathbb{N}$ using only its topology. However,
if we ask that addition should be computable, or that the order relation
should be decidable, etc, we may end up distinguishing a single representation
as the appropriate one. Such questions are linked to the study of
computable model theory. 

\medskip

A possible way to equip a topological space $X$ with an additional
structure that will permit to distinguish a single class of representations
as ``computably admissible'', and that has been used in computable
analysis since early work of Weihrauch \cite{Kreitz1985,Weihrauch1987},
is to consider a numbered basis $(\mathfrak{B},\beta)$ associated
to $X$, i.e. a countable basis $\mathfrak{B}$ for the topology of
$X$ together with a partial surjection $\beta:\subseteq\mathbb{N}\rightarrow\mathfrak{B}$.
Note that the crucial point here is that by fixing a numbering of
a basis we are already choosing the desired notion of computability.
Fixing an abstract basis, as a set and not as a numbered set, would
not be sufficient for this purpose. This is a very natural way to
proceed because, in practice, when working with explicit topological
spaces, there is often an obvious numbering of a basis that stands
out as the correct one to study computability. 

In fact, one can easily remark that in many settings where authors
have used numbered bases, one might as well have considered only numbered
subbases and obtained the same results. Additionally, because it is
important for us to be able to study computability on non-$\text{T}_{0}$
spaces (for instance to include at least all finite topological spaces
in our field of study), we will allow the use of multi-representations. 

This gives the setting of the present study: a topological space equipped
with a numbered subbasis, on which we want to define a multi-representation. 

\medskip

However, even once a numbered subbasis has been fixed on a space,
there still are several, all seemingly natural, but sometimes non-equivalent,
ways to define a multi-representation associated to this numbered
subbasis. 

In the present article, we consider a family of representations introduced
by Spreen in \cite{Spreen2001}, defined thanks to a \emph{strong
inclusion relation}. 

In \cite{Spreen2001}, these representations are used in a context
in which they are computably equivalent to the more commonly used
``standard representation'' of Weihrauch.

Our purpose here is to show that in more general settings, representations
defined thanks to strong inclusion relations do not have to be equivalent
to the standard representation of Weihrauch, and that they can in
some cases be \emph{better behaved} than this standard representation. 

\medskip

Throughout, we fix a set $X$, and denote by $(\mathfrak{B},\beta)$
a numbered subbasis for $X$. This is simply a countable subset of
$\mathcal{P}(X)$ equipped with a partial surjection $\beta:\subseteq\mathbb{N}\rightarrow\mathfrak{B}$.

There are two main approaches that authors have used to define a multi-represen\-tation
$\rho:\subseteq\mathbb{N}^{\mathbb{N}}\rightrightarrows X$ associated
to the numbered subbasis $(\mathfrak{B},\beta)$.

In each case the $\rho$-name of a point $x$ of $X$ is a sequence
$(u_{n})_{n\in\mathbb{N}}$ of $\beta$-names of basic sets which
form a neighborhood basis of $x$. But with this idea, there are two
possible approaches:
\begin{itemize}
\item The sequence $(u_{n})_{n\in\mathbb{N}}$ is asked to contain $\beta$-names
for\emph{ sufficiently many basic} \emph{sets} so as to define a neighborhood
basis of $x$. This representation was first used by Weihrauch in
\cite{Weihrauch1987}. It is particularly important since it was used
by Schröder to prove the characterization theorem of topological spaces
that admit admissible multi-representations. See \cite{Schroeder2001,Schroeder2002}.
It also appears in \cite[Example 2.2]{Kihara2022}. 
\item Or the sequence $(u_{n})_{n\in\mathbb{N}}$ is asked to contain \emph{all
the $\beta$-names for basic sets that contain $x$.} This was first
used in \cite{Kreitz1985} under the name \emph{standard representation}.
This is also the definition of Weihrauch and Grubba \cite{Weihrauch2009ElementaryCT}.
This is now the common approach: see \cite{Hoyrup2016,Schroeder2021}. 
\end{itemize}
In Schröder's dissertation \cite{Schroeder2003}, both approaches
are used, depending on whether the focus is solely continuity (first
approach, see Section 3.1.2 in \cite{Schroeder2003}, which deals
with limit spaces and limit bases), or also computability (second
approach, see Section 4.3.6 in \cite{Schroeder2003}). 

We thus define two multi-representations associated to the numbered
subbasis $(\mathcal{B},\beta)$. They are denoted by $\rho_{\beta}^{\min}:\subseteq\mathbb{N}^{\mathbb{N}}\rightrightarrows X$
and $\rho_{\beta}^{\max}:\subseteq\mathbb{N}^{\mathbb{N}}\rightrightarrows X$,
and are defined by:
\[
\rho_{\beta}^{\min}(f)\ni x\Leftrightarrow\begin{cases}
\forall n\in\text{Im}(f),\,n\in\text{dom}(\beta)\,\&\,x\in\beta(n),\\
\forall B\in\mathfrak{B},\,x\in B\Rightarrow\exists n_{1},...,n_{k}\in\mathbb{N},\,\beta(f(n_{1}))\cap...\cap\beta(f(n_{k}))\subseteq B;
\end{cases}
\]
\[
\rho_{\beta}^{\max}(f)\ni x\iff\text{Im}(f)=\{n\in\text{dom}(\beta),\,x\in\beta(n)\}.
\]

Note that there always exists a translation $\rho_{\beta}^{\max}\le\rho_{\beta}^{\min}$,
witnessed by the identity realizer on Baire space.

Both approaches highlighted above are particular cases of a more general
definition based on a strong inclusion relation, corresponding respectively
to the coarsest and finest reflexive strong inclusion relations.

The idea of replacing set inclusion by a formal relation which need
not be extensional can be traced back to Schreiber and Weihrauch in
\cite{Weihrauch1981}. It is now commonly used in the theory of domain
representation \cite{StoltenbergHansen2008}. 

The use of a formal inclusion relation in relation with numbered bases
was initiated by Dieter Spreen in \cite{Spr98}, in terms of \emph{strong
inclusion relation}s. 
\begin{defiC}
[{\cite[Definition 2.3]{Spr98}}]Let $\mathfrak{B}$ be a subset of
$P(X)$, and $\beta:\subseteq\mathbb{N}\rightarrow\mathfrak{B}$ a
numbering of $\mathfrak{B}$. Let $\mathring{\subseteq}$ be a binary
relation on $\text{dom}(\beta)$. We say that $\mathring{\subseteq}$
is a \emph{strong inclusion relation for $(\mathfrak{B},\beta)$}
if the following hold:
\begin{enumerate}
\item The relation $\mathring{\subseteq}$ is transitive;
\item $\forall b_{1},b_{2}\in\text{dom}(\beta),\,b_{1}\mathring{\subseteq}b_{2}\implies\beta(b_{1})\subseteq\beta(b_{2})$.
\end{enumerate}
\end{defiC}

The general definition of the multi-representation associated to a
numbered subbasis based on a strong inclusion is the following:
\begin{itemize}
\item A sequence $(u_{n})_{n\in\mathbb{N}}$ of $\beta$-names constitutes
a $\rho$-name of $x$ if it contains sufficiently many basic sets,
but \emph{sufficiently many with respect to the strong inclusion. }
\end{itemize}
More precisely, consider a numbered subbasis $(\mathfrak{B},\beta)$. 

We denote by $(\hat{\mathfrak{B}},\hat{\beta})$ the numbered basis
induced by $(\mathfrak{B},\beta)$: $\hat{\mathfrak{B}}$ it the set
of all finite intersections of elements of $\mathfrak{B}$, and $\hat{\beta}$
is the naturally associated numbering: the $\hat{\beta}$-name of
an intersection $B_{1}\cap...\cap B_{n}$ encodes $\beta$-names for
$B_{1}$, ...,$B_{n}$. 

When the basis $(\hat{\mathfrak{B}},\hat{\beta})$ is equipped with
a strong inclusion relation $\mathring{\subseteq}$, we define a multi-representation
$\rho_{\beta}^{\mathring{\subseteq}}:\subseteq\mathbb{N}^{\mathbb{N}}\rightrightarrows X$
by: 
\[
\rho_{\beta}^{\mathring{\subseteq}}(f)\ni x\iff\begin{cases}
\forall b_{1}\in\text{Im}(f),\,b_{1}\in\text{dom}(\hat{\beta})\,\&\,x\in\hat{\beta}(b_{1}),\\
\forall b_{1}\in\text{dom}(\hat{\beta}),\,x\in\hat{\beta}(b_{1})\implies\exists b_{2}\in\text{Im}(f),\,b_{2}\mathring{\subseteq}b_{1}
\end{cases}.
\]

Let $X$ be a set equipped with a numbered subbasis $(\mathcal{B},\beta)$,
and suppose that the induced numbered basis $(\hat{\mathfrak{B}},\hat{\beta})$
admits a strong inclusion relation $\mathring{\subseteq}$. When $x$
is a point of $X$, a subset $A\subseteq\text{dom}(\hat{\beta})$
is called a \emph{strong neighborhood basis for $x$ }if and only
if 
\[
\forall b\in A,\,x\in\hat{\beta}(b);
\]
\[
\forall b_{1}\in\text{dom}(\hat{\beta}),\,x\in\hat{\beta}(b_{1})\implies\exists b_{2}\in A,\,b_{2}\mathring{\subseteq}b_{1}.
\]
Thus the $\rho_{\beta}^{\mathring{\subseteq}}$-name of a point $x$
is a list of $\hat{\beta}$-names that forms a strong neighborhood
basis of $x$.

Note that in general, not all points of $X$ need to have a strong
neighborhood basis with respect to a strong inclusion relation $\mathring{\subseteq}$.
The formula given above for $\rho_{\beta}^{\mathring{\subseteq}}$
correctly defines a multi-representation if and only if every point
admits a strong neighborhood basis, otherwise it is not surjective.
Thus in the present article, we will always suppose that all points
admit strong neighborhood bases. Two natural conditions are sufficient
for all points to admit strong neighborhood bases: 
\begin{itemize}
\item If $\mathring{\subseteq}$ is a reflexive relation. Reflexive strong
inclusions induce equivalence relations, this is often useful, see
\cite{Rauzy2023NewDefs}. 
\item If the basis is a \emph{strong basis }in the sense of Spreen \cite{Spr98}:
if for every two $\hat{\beta}$-names $a$ and $b$ of basic sets
that intersect, and any $x$ in their intersection, there is a $\hat{\beta}$-name
$c$ of a set that contain $x$ and for which $c\mathring{\subseteq}a$
and $c\mathring{\subseteq}b$. 
\end{itemize}
The definition of $\rho_{\beta}^{\mathring{\subseteq}}$ generalizes
the two previous definitions:
\begin{itemize}
\item When we take the strong inclusion to be the actual inclusion relation,
i.e. $b_{1}\mathring{\subseteq}b_{2}\iff\hat{\beta}(b_{1})\subseteq\hat{\beta}(b_{2})$,
$\rho_{\beta}^{\mathring{\subseteq}}$ is exactly $\rho_{\beta}^{\min}$.
Note that this is the coarsest strong inclusion relation. 
\item When we take the strong inclusion to be equality, i.e. $b_{1}\mathring{\subseteq}b_{2}\iff b_{1}=b_{2}$,
$\rho_{\beta}^{\mathring{\subseteq}}$ is exactly $\rho_{\beta}^{\max}$.
Note that equality is the finest reflexive strong inclusion relation. 
\end{itemize}
Note that for every strong inclusion relation $\mathring{\subseteq}$
with respect to which all points admit strong neighborhood bases,
we have $\rho_{\beta}^{\max}\le\rho_{\beta}^{\mathring{\subseteq}}\le\rho_{\beta}^{\min}$,
the identity on Baire space being a realizer for both translations. 

\medskip

The central idea of the present article can be summarized as follows. 

Let $(\mathcal{B},\beta)$ be a numbered basis for a set $X$. 

The correct definition of the representation associated to $\beta$
will often involve giving more information about points than simply
listing names for a neighborhood basis. For instance in a metric space,
we want the name of a point $x$ to give us access to open balls with
small radii that contain a $x$ -if $x$ is isolated, a neighborhood
basis could contain only balls that are explicitly given with big
radii, and we would have to be able to guess that $x$ is isolated
to make use of this name. 

There is a certain amount of ``additional information'' required
to define the correct representation. It is precisely determined by
an appropriate choice of a strong inclusion relation. 

The solution devised by Kreitz and Weihrauch in \cite{Kreitz1985}
is to systematically provide the maximal amount of information that
the basis can provide, giving \emph{all} names of balls that contain
a given point $x$. This corresponds to systematically using the finest
reflexive strong inclusion relation. 

Using a strong inclusion relation different from the finest one gives
a precise quantification of the information that the basis should
provide, offering a better understanding of the representation at
hand. Again, in metric spaces, it is clear that the correct amount
of information required to manipulate a point $x$ is to have balls
with arbitrarily small radii that contain $x$ be given. Listing \emph{all}
the balls with rational radii that contain $x$ obviously involves
listing much more than what is actually necessary.

And the main problem that arises with the Kreitz-Weihrauch method,
where the maximal amount of information is systematically given, is
that as soon as one tries to use non computably enumerable bases,
this amount of information becomes too important, and no point has
a computable name. For instance, when trying to use as a basis of
the topology of $\mathbb{R}$ the set of open intervals with computable
reals as endpoints, the representation $\rho_{\beta}^{\max}$ has
no computable point. But for every reasonable notion of ``computable
equivalence'' of bases, the basis that uses rational intervals is
equivalent to the one that uses intervals with computable reals as
endpoints -just like classically the basis consisting of all open
intervals is equivalent to the one where rational intervals are used. 

Quantifying more precisely the amount of information a basis is supposed
to provide, thanks to a strong inclusion relation, allows to avoid
this caveat. 

\medskip

This article is organized around five main topics, which correspond
respectively to Sections \ref{sec:Admissibility-theorem}, \ref{sec: Set and Subset},
\ref{sec:Metric-spaces-and}, \ref{sec:Semi-decidable strong inclusion }
and \ref{sec:Equivalence-of-bases}. 

Firstly, we show that in any situation, all multi-representations
$\rho_{\beta}^{\min}$, $\rho_{\beta}^{\max}$ and $\rho_{\beta}^{\mathring{\subseteq}}$
are admissible for the topology generated by the basis $\mathfrak{B}$.
They can thus be used interchangeably when focusing on continuity. 

Then, we discuss what happens when considering the multi-representations
associated to a set $X$ and to a subset $A$ of $X$, and investigate
whether we can guarantee that the embedding $A\hookrightarrow X$
will be computable.

We then focus on metric spaces, and on the problem of defining, thanks
to a numbering of open balls, a representation that is equivalent
to the Cauchy representation.

We then consider semi-decidable strong inclusions. We render explicit
the benefit of using a c.e. strong inclusion relation. We also note
that when a basis $\mathcal{B}$ is equipped with a partial numbering
$\beta$, it is always possible to totalize $\beta$ while preserving
the representations $\rho_{\beta}^{\min}$, $\rho_{\beta}^{\max}$
and $\rho_{\beta}^{\mathring{\subseteq}}$, but it might not be possible
to preserve the fact that $\text{dom}(\beta)$ was equipped with a
semi-decidable strong inclusion.

Finally, we study different notions of equivalence of bases. We give
examples of bases that one would naturally expect to be equivalent,
but which yield different representations $\rho^{\max}$. We describe
a notion of \emph{representation-equivalent bases} which guarantees
that two bases yield equivalent representations. 

\medskip

\textbf{Acknowledgements}. I thank Mathieu Hoyrup for a careful reading
of the present article. I would also like to thank Vasco Brattka for
helpful discussions, and Andrej Bauer for relevant references. Finally,
I thank the anonymous referee for many valuable comments and suggestions. 

\section{\label{sec:Preliminaries}Preliminaries }

\subsection{Multi-representations: translation and admissibility }
\begin{defi}
A \emph{multi-representation} \cite{Schroeder2003} of a set $X$
is a partial multi-function: $\rho:\subseteq\mathbb{N}^{\mathbb{N}}\rightrightarrows X$
such that every point of $X$ is the image of some point in $\text{dom}(\rho)$. 
\end{defi}

Note that, extentionally, a partial multi-function: $\rho:\subseteq\mathbb{N}^{\mathbb{N}}\rightrightarrows X$
is nothing but a total function to the power-set of $X$: $\rho:\mathbb{N}^{\mathbb{N}}\rightarrow\mathcal{P}(X)$.
The domain of $\rho$ in this case is $\{f\in\mathbb{N}^{\mathbb{N}},\,\rho(f)\ne\emptyset\}$.
But in terms of the way one uses a multi-representation, it should
not be seen as a function to $\mathcal{P}(X)$. For instance, the
preimage of a subset $A\subseteq X$ by $\rho$ is not defined as
$\{f\in\mathbb{N}^{\mathbb{N}},\,\rho(f)=A\}$, but as $\rho^{-1}(A)=\{f\in\mathbb{N}^{\mathbb{N}},\,\rho(f)\cap A\neq\emptyset\}$.

For $f\in\mathbb{N}^{\mathbb{N}}$, if $x\in\rho(f)$, then $f$ is
a $\rho$\emph{-name} of $x$. 

\begin{defi}
A partial function $H:\subseteq X\rightrightarrows Y$ between multi-represented
sets $(X,\rho_{1})$ and $(Y,\rho_{2})$ is called $(\rho_{1},\rho_{2})$-computable
if there exists a computable function\footnote{Or ``computable functional'' in the sense of Kleene \cite{Kleene1952}
or Grzegorczyk \cite{Grzegorczyk1955}, see \cite{Weihrauch2000}.} $F:\subseteq\mathbb{N}^{\mathbb{N}}\rightarrow\mathbb{N}^{\mathbb{N}}$
defined at least on all $f\in\text{dom}(\rho_{1})$ for which $\rho_{1}(f)\cap\text{dom}(H)\neq\emptyset$
such that 
\[
\forall x\in\text{dom}(H),\,\forall f\in\text{dom}(\rho_{1}),\,x\in\rho_{1}(f)\implies H(x)\in\rho_{2}(F(f)).
\]
\end{defi}

This definition has a very simple interpretation: the multi-function
$H$ is $(\rho_{1},\rho_{2})$-computable if there is a computable
map which, when given a name for a point $x$, produces a name for
the image of $x$ by $H$. 

When $H:\subseteq X\rightarrow Y$ is a partial function between represented
spaces $(X,\rho_{1})$ and $(Y,\rho_{2})$, any partial function $F:\subseteq\mathbb{N}^{\mathbb{N}}\rightarrow\mathbb{N}^{\mathbb{N}}$
which satisfies the condition written in the above definition is called
a \emph{realizer} for $H$. A function is computable if and only if
it has a computable realizer. 
\begin{defi}
If $\rho_{1}$ and $\rho_{2}$ are multi-representations of a set
$X$, we say that \emph{$\rho_{1}$ translates to $\rho_{2}$, }denoted
by $\rho_{1}\le\rho_{2}$, if the identity $\text{id}_{X}$ of $X$
is $(\rho_{1},\rho_{2})$-computable. The multi-representations $\rho_{1}$
and $\rho_{2}$ are called equivalent if each one translates to the
other, this is denoted by $\rho_{1}\equiv\rho_{2}$. 
\end{defi}

The fact that $\rho_{1}$ translates to $\rho_{2}$ can be interpreted
as meaning that the $\rho_{1}$-name of a point in $X$ provides more
information on this point than a $\rho_{2}$-name would. Note that
for representations, $\rho_{1}\le\rho_{2}$ if and only if there exists
a computable $h:\subseteq\mathbb{N}^{\mathbb{N}}\rightarrow\mathbb{N}^{\mathbb{N}}$
so that $\rho_{2}\circ h=\rho_{1}$, while for multi-representations,
$\rho_{1}\le\rho_{2}$ if and only if there exists a computable $h:\subseteq\mathbb{N}^{\mathbb{N}}\rightarrow\mathbb{N}^{\mathbb{N}}$
so that $\rho_{1}\sqsubseteq\rho_{2}\circ h$, i.e. for all $x\in\text{dom}(\rho_{1})$,
$\rho_{1}(x)\subseteq\rho_{2}\circ h(x)$. 

We also have a notion of continuous reduction: 
\begin{defi}
If $\rho_{1}$ and $\rho_{2}$ are multi-representations of a set
$X$, we say that \emph{$\rho_{1}$ continuously translates to $\rho_{2}$}
if the identity $\text{id}_{X}$ of $X$ admits a continuous realizer.
This is denoted by $\rho_{1}\le_{t}\rho_{2}$, where the subscript
$t$ stands for \emph{topological}. 

The multi-representations $\rho_{1}$ and $\rho_{2}$ are \emph{continuously
equivalent} when each one continuously translates to the other. This
is denoted by $\rho_{1}\equiv_{t}\rho_{2}$. 
\end{defi}

When $X$ is equipped with a multi-representation $\rho:\subseteq\mathbb{N}^{\mathbb{N}}\rightarrow X$,
we can naturally equip $X$ with a topology, the final topology $\mathcal{T}_{\rho}$
of the multi-representation. The topology $\mathcal{T}_{\rho}$ is
defined as follows \cite[p. 53]{Schroeder2003}: a set $A$ is open
in $X$ if and only if $A=\rho(\rho^{-1}(A))$ and $\rho^{-1}(A)$
is open in the Baire space topology (more precisely: open in the topology
on $\text{dom}(\rho)$ induced by the topology of Baire space).

Note that when $\rho$ is a function, $A=\rho(\rho^{-1}(A))$ is automatically
satisfied, since $\rho$ is surjective. Generalizing this to multi-functions,
we have render explicit this condition. 

In this paper, as we focus on second countable topological spaces,
the topology of each space we consider is determined by converging
sequences, and thus by the \emph{limit space} induced by the topology.
However, when working with multi-representations, we need the notion
of \emph{sequentially continuous multi-representation} even when focusing
solely on second countable spaces. 
\begin{defiC}
[{\cite[Section 2.4.4]{Schroeder2003}}]A multi-function $F:X\rightrightarrows Y$
is called \emph{sequentially continuous} if for every sequence $(x_{n})\in X^{\mathbb{N}}$
that converges to a point $x_{\infty}$, every sequence $(y_{n})\in Y^{\mathbb{N}}$
with $y_{n}\in F(x_{n})$ for all $n$, and every $y_{\infty}\in F(x_{\infty})$,
$(y_{n})$ converges to $y_{\infty}$. 
\end{defiC}

The reason why we need this definition even in a sequential setting
is the following. Consider the degenerate multi-function $F:\mathbb{N}^{\mathbb{N}}\rightrightarrows\mathbb{N}^{\mathbb{N}}$
which maps every point to $\mathbb{N}^{\mathbb{N}}$. The preimage
of any open set by $F$ is open in Baire space, and the preimage of
any closed set is closed (thus $F$ is both \emph{lower} and \emph{upper
semi--continuous}). However, there is no continuous translation between
$F$ (seen as a multi-representation) to the usual identity representation
of Baire space. This comes from the fact that while $F$ is in a sense
continuous, the final topology of $F$ is still coarser than the Baire
space topology. This cannot happens for single valued functions: a
function $F:\mathbb{N}^{\mathbb{N}}\rightarrow\mathbb{N}^{\mathbb{N}}$
is continuous if and only if the final topology of $F$ is finer than
the topology of Baire space. 

\begin{defiC}
[\cite{Schroeder2003}]Let $(X,\mathcal{T})$ be a sequential topological
space. A multi-representation $\rho:\subseteq\mathbb{N}^{\mathbb{N}}\rightrightarrows X$
of $(X,\mathcal{T})$ is called \emph{admissible} if it is a maximal
sequentially continuous multi-representation: $\rho$ is sequentially
continuous, and for any sequentially continuous multi-representation
$\phi:\subseteq\mathbb{N}^{\mathbb{N}}\rightrightarrows X$ of $X$
we have $\phi\le_{t}\rho.$
\end{defiC}

Note that the following lemma guarantees that when we consider a $\text{T}_{0}$-space,
the above supremum can be taken only over continuous representations.
\begin{lem}
[{\cite[Lemma 2.4.6]{Schroeder2003}}] \label{lem: T0 and multifunction}If
a multi-function $F:X\rightrightarrows Y$ towards a $\text{T}_{0}$
space is sequentially continuous, it is in fact a single valued function. 
\end{lem}

\subsection{\label{subsec:Sierpi=000144ski-representation-of}Sierpi\'{n}ski
representation of open sets}

Associated to every multi-represented space $(X,\rho)$ is a representation
of the open sets of the final topology of $\rho$, called the Sierpi\'{n}ski
representation \cite{Schroeder2003}, denoted by $[\rho\rightarrow\rho_{\mathfrak{Si}}]$
as in \cite{Schroeder2003}. See also \cite{Pauly2016}. Note that
even when $\rho$ is a multi-representation, $[\rho\rightarrow\rho_{\mathfrak{Si}}]$
remains an actual representation. 

This representation is defined as follows: $[\rho\rightarrow\rho_{\mathfrak{Si}}](p)=O$
if and only if $p$ encodes a pair $(n,q)$, where $n\in\mathbb{N}$
and $q\in\mathbb{N}^{\mathbb{N}}$, and $n$ is the code of a Type
2 Turing machine which, when run on input $f\in\text{dom}(\rho)$
using $q$ as an oracle, will stop if and only if $\rho(f)\in O$.
For more details, see \cite{Schroeder2003}.

A subset of $X$ is called \emph{c.e. open} if it is a computable
point of the Sierpi\'{n}ski representation. In other words, it is
simply a $\rho$-semi-decidable set. 

\subsection{Numberings}

A numbering of a set $X$ is a partial surjection $\nu:\subseteq\mathbb{N}\rightarrow X$.
This can equivalently be seen as a representation whose domain is
a subset of the set of constant sequences, and thus the notion of
computable function and the order $\le$ defined in the previous section
can also be applied to numberings.

We say that $X$ is \emph{$\nu$-computably enumerable }if there is
a c.e. subset $A$ of $\mathbb{N}$ such that $X=\nu(A)$. This is
one of several possible notions that formalize the idea of being ``effectively
countable'', and probably the more commonly considered one. 

Supposing that a numbering $\nu$ has domain $\mathbb{N}$, or has
a c.e. domain, in particular implies that the numbered set is $\nu$-computably
enumerable. Applied to numbered bases, this gives a notion of effective
second countability. 

\subsection{Basis induced by a subbasis and induced strong inclusion }

Fix a set $X$, and denote by $(\mathfrak{B},\beta)$ a numbered subbasis
for $X$. The \emph{induced numbered basis }is a pair \emph{$(\hat{\mathfrak{B}},\hat{\beta})$}:
$\hat{\mathfrak{B}}$ is the set of finite intersections of elements
of $\mathfrak{B}$, and $\hat{\beta}$ is a numbering defined as follows.
Denote by $\Delta$ the standard numbering of finite subsets of $\mathbb{N}$.
Then put:
\[
\text{dom}(\hat{\beta})=\{n\in\mathbb{N},\,\Delta_{n}\subseteq\text{dom}(\beta)\};
\]
\[
\forall n\in\text{dom}(\hat{\beta}),\,\hat{\beta}(n)=\bigcap_{B\in\beta(\Delta_{n})}B.
\]

When $(\mathfrak{B},\beta)$ is equipped with a strong inclusion $\mathring{\subseteq}$,
this strong inclusion can naturally be extended to $(\hat{\mathfrak{B}},\hat{\beta})$,
as follows:
\[
\forall n_{1},n_{2}\in\text{dom}(\hat{\beta}),n_{1}\mathring{\subseteq}n_{2}\iff\forall k\in\Delta_{n_{2}},\thinspace\exists p\in\Delta_{n_{1}},\,p\mathring{\subseteq}k.
\]
One easily check that this extension preserves the conditions of being
a strong inclusion relation, and that if $\mathring{\subseteq}$ was
reflexive, the extension remains reflexive. 
\begin{rem}
In the present article, we often consider a numbered subbasis $(\mathfrak{B},\beta)$
which induces a basis $(\hat{\mathfrak{B}},\hat{\beta})$, and a strong
inclusion relation $\mathring{\subseteq}$ for $(\hat{\mathfrak{B}},\hat{\beta})$
-and not for $(\mathfrak{B},\beta)$. The reason for this is as follows.
As shown above, every strong inclusion relation for $(\mathfrak{B},\beta)$
induces a strong inclusion relation on $(\hat{\mathfrak{B}},\hat{\beta})$.
However, not every strong inclusion for $(\hat{\mathfrak{B}},\hat{\beta})$
needs to come from a strong inclusion of $(\mathfrak{B},\beta)$.
In particular, the ``set inclusion strong relation'' $n_{1}\mathring{\subseteq}n_{2}\iff\beta(n_{1})\subseteq\beta(n_{2})$
on $(\mathfrak{B},\beta)$ does not induce the set inclusion strong
relation on $(\hat{\mathfrak{B}},\hat{\beta})$. And the relation
$n_{1}\mathring{\subseteq}n_{2}\iff\hat{\beta}(n_{1})\subseteq\hat{\beta}(n_{2})$
for $(\hat{\mathfrak{B}},\hat{\beta})$ does not have to be induced
by a strong inclusion of $(\mathfrak{B},\beta)$. 

Our setting is thus made more general by allowing any strong inclusion
relation on $(\hat{\mathfrak{B}},\hat{\beta})$, and not only those
that come from strong inclusions of the subbasis. 
\end{rem}

\subsection{Computable topological spaces }

We now introduce here the notion of ``computable topological space''
that was introduced by Weihrauch and Grubba in \cite{Weihrauch2009ElementaryCT}.
It is a special case of a definition of Bauer: \cite[Definition 5.4.2]{Bauer2000}. 

We want to note here that the term ``computable topological space''
is not an appropriate name for this notion, since it is a notion of
\emph{computable basis}, and not the definition of a computable topological
space. Furthermore, it is only one amongst several possible notions
of computable basis, and not the most general one one can think of.
For instance it does not apply to all non-computably separable computable
metric spaces. 

Denote by $W_{i}=\text{dom}(\varphi_{i})$ the usual numbering of
r.e. subsets of $\mathbb{N}$. 
\begin{defiC}
[\cite{Weihrauch2009ElementaryCT}]\label{def:Usual Definition }A
\emph{``computable topological space''} is a triple $(X,\mathcal{B},\beta)$,
where $X$ is a set, $\mathcal{B}$ is a topological basis on $X$
that makes of it a $T_{0}$ space, and $\beta:\mathbb{N}\rightarrow\mathcal{B}$
is a total surjective numbering of $\mathcal{B}$, for which there
exists a computable function $f:\mathbb{N}^{2}\rightarrow\mathbb{N}$
of such that for any $i$, $j$ in $\mathbb{N}$: 
\[
\beta(i)\cap\beta(j)=\underset{k\in W_{f(i,j)}}{\bigcup}\beta(k).
\]
\end{defiC}

Note that the requirement that $\beta$ be total can be read as imposing
that $(X,\mathcal{B},\beta)$ be \emph{computably second countable}.
This is a requirement one might want to do away with. 

A possible better name for the notion above would be that of a \emph{Lacombe
basis}. In particular, if we do not ask $\beta$ to be total, and
if we do not suppose that $X$ will be $T_{0}$, then the conditions
imposed on the basis are exactly \emph{the necessary and sufficient
conditions in order for the Lacombe sets}\footnote{Lacombe sets are computable union of basic open sets, this name goes
back to Lachlan \cite{Lachlan1964} and Moschovakis \cite{Moschovakis1964}.} \emph{to form a computable topology}: so that finite intersection
and computable unions be computable. 

Associated to a ``computable topological space'' is a representation
of open sets: 
\[
\rho_{(\mathfrak{B},\beta)}(f)=\underset{\{n,\,\exists p\in\mathbb{N},\,f(p)=n+1\}}{\bigcup}\beta(n).
\]
(In the above, if $f=0^{\omega}$, then it is a name of the empty
set.) 

The condition of Definition \ref{def:Usual Definition } are sufficient
in order for finite intersections and countable unions to be computable
for $\rho_{(\mathfrak{B},\beta)}$, and, again, removing the condition
that $X$ be $T_{0}$ and that $\beta$ be total, we obtain necessary
and sufficient conditions. 

The representation $\rho_{\beta}^{\max}$ has been up to now often
associated to the above definition of a computable basis. One of the
purposes of this article is to show that the notion of ``computable
topological space'' described above, while very relevant to the study
of the associated representation $\rho_{(\mathfrak{B},\beta)}$ of
open sets, is not relevant to the study of the representations $\rho_{\beta}^{\min}$,
$\rho_{\beta}^{\max}$ and $\rho_{\beta}^{\mathring{\subseteq}}$
associated to a numbered basis. In particular, none of the conditions
of Definition \ref{def:Usual Definition } are useful in showing that
$\rho_{\beta}^{\min}$, $\rho_{\beta}^{\max}$ and $\rho_{\beta}^{\mathring{\subseteq}}$
are admissible representations (or admissible multi-representations,
if we allow non $T_{0}$-spaces). 

\subsection{\label{subsec:Computable-metric-spaces}Computable metric spaces }

The following is the common definition for computable metric spaces.
\begin{defiC}
[\cite{Weihrauch2003,Brattka2003}] A\emph{ computable metric space
}(CMS) is a quadruple $(X,A,\nu,d)$, where $(X,d)$ is a metric space,
$A$ is a countable and dense subset of $X$, $\nu:\mathbb{N}\rightarrow A$
is a total numbering of $A$, and such that the metric $d:A\times A\rightarrow\mathbb{R}_{c}$
is $(\nu\times\nu,c_{\mathbb{R}})$-computable. 
\end{defiC}

The following older definition is in fact more general. It can be
seen as the generalization of Moschovakis' notion of a \emph{recursive
metric space, }that comes from \cite{Moschovakis1964} and\emph{ }which
concerns only countable spaces, to a Type 2 setting. (Indeed, not
every recursive metric space in the sense of \cite{Moschovakis1964}
is a computable metric space.)
\begin{defiC}
[\cite{Hertling1996}] A \emph{non-necessarily effectively
separable computable metric space} is a quadruple $(X,A,\nu,d)$,
where $(X,d)$ is a metric space, $A$ is a countable and dense subset
of $X$, $\nu:\subseteq\mathbb{N}\rightarrow A$ is a partial numbering
of $A$, and such that the metric $d:A\times A\rightarrow\mathbb{R}_{c}$
is $(\nu\times\nu,c_{\mathbb{R}})$-computable. 
\end{defiC}

The existence of constructively non separable metric spaces has been
known for a long time. In \cite{Slisenko1972}, Slisenko constructs,
in the context of countable numbered sets, a non-computably separable
metric space which cannot be embedded in a computably separable one.
A more conceptual proof of this result was given by Weihrauch in \cite{Weihrauch2013}:
any CMS satisfies a strong form of effective regularity, known as
$\text{SCT}_{3}$, and non-separable computable metric spaces, while
always computably regular ($\text{CT}_{3}$), do not have to be strongly
regular. But because $\text{SCT}_{3}$ is a property inherited by
subsets, the example of a non-computably separable metric space that
satisfies $\text{CT}_{3}$ but not $\text{SCT}_{3}$ \cite[Example 5.4]{Weihrauch2013}
is a non-computably separable metric space that does not embed into
a CMS. 

Non-computably separable spaces naturally arise, see the author's
paper \cite{RauzyV}, which motivated the present note. Note also
that, when applying the Schröder Metrization theorem \cite{Schroeder1998,Weihrauch2013}
to a represented space, one builds a computable metric, but there
is no guarantee that the resulting space will indeed be computably
separable, and an actual computable metric space. 

Denote by $c_{\mathbb{Q}}$ the usual numbering of $\mathbb{Q}$,
which is total. Associated to a (non-necessarily effectively separable)
computable metric space $(X,A,\nu,d)$ is a numbering $\beta$ of
open balls centered at points of $A$: 
\[
\text{dom}(\beta)=\{\langle n,m\rangle\in\mathbb{N},\,n\in\text{dom}(\nu),\,c_{\mathbb{Q}}(m)>0\},
\]
\[
\forall\langle n,m\rangle\in\text{dom}(\beta),\,\beta(\langle n,m\rangle)=B(\nu(n),c_{\mathbb{Q}}(m)).
\]
Note that when $\nu$ is total, $\beta$ has a recursive domain, we
can then suppose that it is total. 

The are two natural strong inclusion relations on metric spaces: 
\[
\langle n_{1},m_{1}\rangle\mathring{\subseteq}\langle n_{2},m_{2}\rangle\iff d(\nu(n_{1}),\nu(n_{2}))+c_{\mathbb{Q}}(m_{1})<c_{\mathbb{Q}}(m_{2});
\]
\[
\langle n_{1},m_{1}\rangle\mathring{\subseteq}\langle n_{2},m_{2}\rangle\iff d(\nu(n_{1}),\nu(n_{2}))+c_{\mathbb{Q}}(m_{1})\le c_{\mathbb{Q}}(m_{2}).
\]
The first one has the advantage of being semi-decidable, the second
one of being reflexive, both have found use in the literature \cite{Weihrauch1981,Spr98}.
But in fact, these two different relations induce the same notion
of strong neighborhood basis, and thus they can be studied interchangeably
when studying the representation $\rho_{\beta}^{\mathring{\subseteq}}$. 

Metric spaces come equipped with the Cauchy representation $\rho_{Cau}$:
\begin{multline*} \text{dom}(\rho_{Cau}) = \{ p \in \text{dom}(\nu)^{\mathbb{N}},\ \forall i > j, \\ d(\nu(p(i)), \nu(p(j))) < 2^{-j},\ \exists x \in X,\ x = \lim_{i \to +\infty} \nu(p(i)) \} \end{multline*}
\[
\forall p\in\text{dom}(\rho_{Cau}),\,\rho_{Cau}(p)=\lim_{i\rightarrow+\infty}\nu(p(i)).
\]

\section{\label{sec:Admissibility-theorem}Admissibility theorem }

In this section, we prove that all multi-representations $\rho_{\beta}^{\min}$,
$\rho_{\beta}^{\max}$ and $\rho_{\beta}^{\mathring{\subseteq}}$
are admissible. 
\begin{prop}
For any numbered set $(\mathcal{B},\beta)$ of subsets of a set $X$,
and any strong inclusion relation $\mathring{\subseteq}$ for $(\mathcal{B},\beta)$,
the three multi-representations $\rho_{\beta}^{\min}$, $\rho_{\beta}^{\max}$
and $\rho_{\beta}^{\mathring{\subseteq}}$ are equivalent modulo continuous
translations. 
\end{prop}

\begin{proof}
As remarked in the introduction, we always have $\rho_{\beta}^{\max}\le\rho_{\beta}^{\mathring{\subseteq}}\le\rho_{\beta}^{\min}.$We
thus prove that $\rho_{\beta}^{\min}\le_{t}\rho_{\beta}^{\max}$. 

With a powerful enough oracle, the inclusion can be decided on $\text{dom}(\beta)$
(i.e. the relation $R$ defined by $nRm\iff\beta(n)\subseteq\beta(m)$).
Also, a powerful enough oracle can enumerate $\text{dom}(\beta)$.
With such an oracle, we can, given the $\rho_{\beta}^{\min}$-name
of a point $x$, enumerate in parallel all names of balls that contain
a ball that contains $x$, this will precisely give a $\rho_{\beta}^{\max}$-name
of $x$. 
\end{proof}
\begin{thm}
\label{thm:Admissibility }For any numbered set $(\mathcal{B},\beta)$
of subsets of a set $X$, and any strong inclusion relation $\mathring{\subseteq}$
for $(\hat{\mathcal{B}},\hat{\beta})$ for which points have strong
neighborhood bases, the three multi-representations $\rho_{\beta}^{\min}$,
$\rho_{\beta}^{\max}$ and $\rho_{\beta}^{\mathring{\subseteq}}$
are admissible with respect to the topology of $X$ generated by $\mathcal{B}$
(as a subbasis). 
\end{thm}

\begin{proof}
This is an immediate corollary of the previous result, together with
the theorem of Schröder which states that $\rho_{\beta}^{\min}$ is
always admissible. See in particular in \cite{Schroeder2003}: Proposition
3.1.6, for the case of limit spaces, and Lemma 3.1.10 for the transfer
of this result to topological spaces. 
\end{proof}
Note that in cases where the topology generated by $(\mathcal{B},\beta)$
satisfies the $T_{0}$ axiom, and $\rho_{\beta}^{\min}$, $\rho_{\beta}^{\max}$
and $\rho_{\beta}^{\mathring{\subseteq}}$ are representations instead
of multi-representations, the above result follows from classical
results of Weihrauch \cite{Weihrauch1987a,Weihrauch2000} who shows
that $\rho_{\beta}^{\max}$ is admissible. And in fact, one can also
deduce Theorem \ref{thm:Admissibility } from Weihrauch's result,
together with the following lemma, suggested to us by an anonymous
referee. Recall that on a topological space $(X,\mathcal{T})$, we
define an equivalence relation $\sim$ by: 
\[
a\sim b\iff\forall O\in\mathcal{T},\,a\in O\iff b\in O.
\]
 The quotient $X/\sim$, that comes equipped with the quotient topology,
is called the Kolmogorov quotient of $(X,\mathcal{T})$. 
\begin{lem}
Let $(X,\rho)$ be a multi-represented space with final topology $\mathcal{T}$.
Consider the Kolmogorov quotient $X/\sim$ of $X$, and the quotient
representation $\tilde{\rho}$. Then $\tilde{\rho}$ is admissible
if and only if $\rho$ is. 
\end{lem}

\begin{proof}
Suppose first that $\tilde{\rho}$ is admissible. 

Let $\tau$ be any sequentially continuous multi-representation of
$X$. Because the quotient map $X\overset{q}{\rightarrow}X/\sim$
is sequentially continuous, $q\circ\tau$ is a sequentially continuous
multi-function whose codomain is $\text{T}_{0}$, by Lemma \ref{lem: T0 and multifunction},
$q\circ\tau$ is actually single valued. It is a continuous representation
of $X/\sim$, and thus $q\circ\tau\le_{t}\tilde{\rho}$, i.e. there
is a continuous function $h$ so that $q\circ\tau\circ h=\tilde{\rho}$
on $\text{dom}(\rho)$. 

This indicates that for every $x\in\text{dom}(\rho)$, every $y\in\tau\circ h(x)$
and every $z$ in $\rho(x)$, $y\sim z$. But by definition of the
final topology of $\rho$, if $y\sim z$ and $z\in\rho(x)$, then
$y\in\rho(x)$. And thus $\tau\circ h(x)\subseteq\rho(x)$ for every
$x$, which indicates that $\tau\circ h\le_{t}\rho$.

Conversely, if $\rho$ is admissible, $\tilde{\rho}$ is also admissible.
Indeed, let $\tilde{\tau}:\subseteq\mathbb{N}^{\mathbb{N}}\rightarrow X/\sim$
be a continuous representation of $X/\sim$. Then $q^{-1}\circ\tilde{\tau}$
is a sequentially continuous multi-representation of $X$. By admissibility
of $\rho$, $q^{-1}\circ\tilde{\tau}\le_{t}\rho$. And $q\circ q^{-1}\circ\tilde{\tau}\le_{t}q\circ\rho$,
i.e. $\tilde{\tau}\le_{t}\tilde{\rho}$. 
\end{proof}

\section{\label{sec: Set and Subset}Compatibility of the multi-representations
of a set and of a subset }

If $(X,\rho)$ is a multi-represented set, and if $A$ is a subset
of $X$, we naturally define a multi-representation $\rho_{\vert A}$
of $A$, \emph{the restriction of $\rho$ to $A$, }by the following:
\[
\text{dom}(\rho_{\vert A})=\{f\in\text{dom}(\rho),\,\rho(f)\cap A\ne\emptyset\};
\]
\[
\forall f\in\text{dom}(\rho_{\vert A}),\,\rho_{\vert A}(f)=\rho(f)\cap A.
\]

\begin{prop}
\label{prop: restriction rpz embedding computable }If $A$ is a subset
of $X$ and $\rho$ is a multi-representation of $X$, then the embedding
$A\hookrightarrow X$ is always $(\rho_{\vert A},\rho)$-computable,
and the identity on Baire space is a realizer. 
\end{prop}

\begin{proof}
Immediate. 
\end{proof}
We now consider a numbered subbasis $(\mathfrak{B},\beta)$ for $X$
equipped with a strong inclusion relation $\mathring{\subseteq}$.
We thus have three multi-representations $\rho_{\beta}^{\min}$, $\rho_{\beta}^{\max}$
and $\rho_{\beta}^{\mathring{\subseteq}}$ of $X$. 

We can now consider a numbered subbasis $(\mathfrak{A},\alpha)$ for
$A$, defined by restriction of $(\mathfrak{B},\beta)$:
\[
\text{dom}(\alpha)=\text{dom}(\beta);
\]
\[
\forall n\in\text{dom}(\beta),\,\alpha(n)=A\cap\beta(n).
\]
In fact, $(\mathfrak{A},\alpha)$ is naturally equipped with the same
strong inclusion relation as $\beta$. Indeed, the condition $n\mathring{\subseteq}m\implies\alpha(n)\subseteq\alpha(m)$
is always valid, since 
\[
\forall n,m\in\text{dom}(\beta),\,n\mathring{\subseteq}m\implies\beta(n)\subseteq\beta(m)\implies A\cap\beta(n)\subseteq A\cap\beta(m).
\]
(Note that the new $\mathring{\subseteq}$ remains reflexive, and
points still admit strong neighborhood bases, if this was the case
for the original $\mathring{\subseteq}$.)

We then have three multi-representations of $A$: $\rho_{\alpha}^{\min}$,
$\rho_{\alpha}^{\max}$ and $\rho_{\alpha}^{\mathring{\subseteq}}$. 

Both $\rho_{\alpha}^{\max}$ and $\rho_{\alpha}^{\mathring{\subseteq}}$
are well behaved with respect to the inclusion $A\hookrightarrow X$,
but $\rho_{\alpha}^{\min}$ is not. This is what we show now. 
\begin{prop}
We have: \textup{$\rho_{\alpha}^{\max}\equiv(\rho_{\beta}^{\max})_{\vert A}$
and $\rho_{\alpha}^{\mathring{\subseteq}}\equiv(\rho_{\beta}^{\mathring{\subseteq}})_{\vert A}$.} 
\end{prop}

\begin{proof}
Left to the reader. 
\end{proof}
\begin{cor}
The embedding\textup{ $A\hookrightarrow X$ is always $(\rho_{\alpha}^{\max},\rho_{\beta}^{\max})$-computable
and $(\rho_{\alpha}^{\mathring{\subseteq}},\rho_{\beta}^{\mathring{\subseteq}})$-computable.}
\end{cor}

\begin{proof}
By Proposition \ref{prop: restriction rpz embedding computable }. 
\end{proof}
\begin{prop}
\label{prop:The-embedding-not-computable min rpz }The embedding\textup{
$A\hookrightarrow X$ is not always $(\rho_{\alpha}^{\min},\rho_{\beta}^{\min})$-computable. }
\end{prop}

The proof of this proposition uses a construction that appears in
Section \ref{sec:The-sufficiently-many-approach and metric spaces }
and is postponed to this section. 
\begin{cor}
It is not always the case that $\rho_{\alpha}^{\min}\equiv(\rho_{\beta}^{\min})_{\vert A}$. 
\end{cor}

\begin{proof}
This follows by Proposition \ref{prop: restriction rpz embedding computable }. 
\end{proof}

\section{\label{sec:Metric-spaces-and}Metric spaces and the Cauchy representation}

Recall from Section \ref{subsec:Computable-metric-spaces} that on
a non-computably separable computable metric space, we have four representations:
$\rho_{\beta}^{\max}$, $\rho_{\beta}^{\min}$, $\rho_{\beta}^{\mathring{\subseteq}}$
(where $\mathring{\subseteq}$ is the strong inclusion that comes
from the metric) and $\rho_{Cau}$, the Cauchy representation. 

In this section, we establish the following result: 
\begin{thm}
The following reductions hold on any computable metric space: 
\[
\rho_{\beta}^{\max}\equiv\rho_{Cau}\equiv\rho_{\beta}^{\mathring{\subseteq}}\le\rho_{\beta}^{\min},
\]
 and $\le$ can sometimes be strict. 

The following reductions hold on any non-computably separable computable
metric space:
\[
\rho_{\beta}^{\max}\le\rho_{Cau}\equiv\rho_{\beta}^{\mathring{\subseteq}}\le\rho_{\beta}^{\min},
\]
 and each $\le$ can sometimes be strict. 
\end{thm}

Note that the positive statements about computable metric spaces in
the above can be gathered from \cite{Kreitz1985,Weihrauch1987,Weihrauch2000,Spreen2001}. 

\subsection{\label{sec:The-sufficiently-many-approach and metric spaces }The
\textquotedblleft sufficiently many basic sets\textquotedblright{}
approach and metric spaces}

Let $(X,A,\nu,d)$ be a non-necessarily computably separable computable
metric space, and denote by $\beta$ the numbering of open balls with
rational radii associated to $(X,A,\nu,d)$. 
\begin{prop}
We have $\rho_{Cau}\le\rho_{\beta}^{\min}$. 
\end{prop}

\begin{proof}
The $\rho_{Cau}$ name of a point $x$ is a sequence (of names) of
points $(u_{n})_{n\in\mathbb{N}}$ that converges towards $x$ at
exponential speed. A sequence of names for the balls $(B(u_{n},2^{-n}))_{n\in\mathbb{N}}$
is then a $\rho_{\beta}^{\min}$-name of $x$. 
\end{proof}
It was already remarked by Weihrauch in \cite[Section 3.4,  Therorem 18]{Weihrauch1987}
that the existence of isolated point could be a problem for the representation
$\rho_{\beta}^{\min}$. We show:
\begin{prop}
\label{prop: discrete and not discrete}The representation $\rho_{\beta}^{\min}$
does not have to be equivalent to the Cauchy representation of $X$,
even if $X$ is computably separable. 
\end{prop}

The idea is the following. We build a computable metric space which
is discrete in some places and not discrete in others. For a point
$x$ which is isolated, if $n$ is a name of a ball $B(x,r)$ which
contains only $x$, i.e. $B(x,r)=\{x\}$, then $(n,n,n,n,...)$ is
a $\rho_{\beta}^{\min}$-name for $x$. However, to construct a Cauchy
name for $x$ starting from this $\rho_{\beta}^{\min}$-name, one
has to be able to computably understand that $x$ is isolated, and
that the ball $B(x,r)$ determines $x$ uniquely. We choose a space
where this is not possible. 

Note that this example is precisely based upon a case where the strong
inclusion relation is different from the actual inclusion relation.
Indeed, with the notations above, we have:
\begin{itemize}
\item For any $m$ such that $x\in\beta(m)$, then $\beta(n)\subseteq\beta(m)$. 
\item It is not true that for any $m$ such that $x\in\beta(m)$, then $n\mathring{\subseteq}m$.
Indeed, this fails if $m$ is a name for $B(x,r/2)$ (i.e. a name
where the radius \emph{explicitly given} is $r/2$). 
\end{itemize}
\begin{proof}
We take a certain subset of $\mathbb{R}$. Denote by $K$ the halting
set: $K=\{n\in\mathbb{N},\,\varphi_{n}(n)\downarrow\}$. Consider
the union 
\[
A=\bigcup_{n\in K}[n-1/2,n+1/2]\cup\bigcup_{n\notin K}\{n\}.
\]
Thus $A$ is discrete in some places and not discrete in others. $A$
admits a dense and computable sequence: the set of natural numbers,
together with the set of rationals in $[n-1/2,n+1/2]$, for each $n$
in $K$, which is a r.e. set because $K$ is. 

The usual metric of $\mathbb{R}$ remains computable when restricted
to the set of rationals in $A$. Thus $A$ is a computable metric
space. Denote by $\beta$ the numbering of open intervals of $A$
induced by that of $\mathbb{R}$: if $\gamma$ is a numbering of open
intervals of $\mathbb{R}$ that have rational endpoints, put 
\[
\beta(n)=\gamma(n)\cap A.
\]

In this setting, for $n\notin K$, denote by $m$ a $\gamma$-name
of the basic set $]n-1/4,n+1/4[$. Then the constant sequence $(m,m,m,m...)$
constitutes a valid $\rho_{\beta}^{\min}$-name of $n$. Suppose there
is a Type-2 machine $T$ that on input the $\rho_{\beta}^{\min}$-name
of a point of $A$ transforms it in a $\rho_{Cau}$-name of this point.
Then we can enumerate numbers $n$ for which the $\rho_{Cau}$-name
produced by $T$ when given as input a constant sequence as above
gives a precision better than $1/4$. This gives precisely an enumeration
of $K^{c}$. This is a contradiction. 
\end{proof}
We finally use the construction above to prove Proposition \ref{prop:The-embedding-not-computable min rpz }.
\begin{proof}
[Proof of Proposition \ref{prop:The-embedding-not-computable min rpz }]In
the construction that appears above in the proof of Proposition \ref{prop: discrete and not discrete},
we build a metric space $A$ equipped with a numbered basis with numbering
$\beta$ and where $\rho_{\beta}^{\min}\not\le\rho_{Cau}$. The constructed
set is a subset of $\mathbb{R}$, and it is easy to see that in this
case the Cauchy numbering $\rho_{Cau}$ on $A$ is the restriction
of the Cauchy numbering on $\mathbb{R}$, which is itself equivalent
to $\rho_{\gamma}^{\min}$, where $\gamma$ denotes the numbering
of open intervals of $\mathbb{R}$ with rational endpoints. And thus
we have: $\rho_{\beta}^{\min}\not\le\rho_{Cau}$, $\rho_{Cau}\equiv(\rho_{\gamma}^{\min})_{\vert A}$,
and thus $\rho_{\beta}^{\min}\not\le(\rho_{\gamma}^{\min})_{\vert A}$.
Finally this directly implies that the embedding $A\hookrightarrow\mathbb{R}$
is not $(\rho_{\beta}^{\min},\rho_{\gamma}^{\min})$-computable. 
\end{proof}

\subsection{The representation associated to \textquotedblleft all names of basic
sets\textquotedblright{} and non-computably separable metric spaces}

Let $(X,A,\nu,d)$ denote again a non-necessarily computably separable
computable metric space, and denote by $\beta$ the numbering of open
balls with rational radii associated to $(X,A,\nu,d)$. 
\begin{prop}
We have $\rho_{\beta}^{\max}\le\rho_{Cau}$. 
\end{prop}

\begin{proof}
This is very simple: a $\rho_{\beta}^{\max}$-name contains names
of balls of arbitrarily small radius. Given a $\rho_{\beta}^{\max}$-name
of a point $x$, a blind search in this name for a $2^{-n}$ good-approximation
of $x$ will always terminate. 
\end{proof}
The following is well known: it shows that the definition of \emph{computable
topological space }as introduced in \cite{Weihrauch2009ElementaryCT}
(see Definition \ref{def:Usual Definition }) is coherent with the
definition of computable metric space as used since \cite{Weihrauch2003}.
It is for instance recalled in \cite[Example 9.4.16]{Schroeder2021}. 
\begin{prop}
As soon as $(X,A,\nu,d)$ has a dense and computable sequence, for
the numbering $\beta$ of open balls given by rational radii and centers
in the dense sequence, one also has $\rho_{Cau}\equiv\rho_{\beta}^{\max}$.
\end{prop}

Note however, as we will discuss in Section \ref{sec:Equivalence-of-bases},
that the above proposition holds only for certain choices of a numbering
of open balls.

Finally, if $(X,A,\nu,d)$ does not have a dense and computable sequence,
it does not have to have a computably enumerable basis of open balls.
In this case, we have:
\begin{prop}
If $\beta$ is the natural numbering of open balls with rational radii
in a non-computably separable computable metric space, it is possible
that $\rho_{Cau}\not\le\rho_{\beta}^{\max}$. And it is possible that
$\rho_{\beta}^{\max}$ defines no computable point. 
\end{prop}

\begin{proof}
We give a very simple example. Consider the set $K^{c}$, the complement
of the halting set, which is not r.e.. The numbering of $K^{c}$ is
the numbering induced by the identity on $\mathbb{N}$. Take the usual
metric of $\mathbb{N}$, $d(i,j)=\vert i-j\vert$, and the associated
numbering $\beta$ of balls with rational radii and centers in $K^{c}$. 

In this setting, no point of $K^{c}$ has admits a computable $\rho_{\beta}^{\max}$-name. 

Indeed, a computable enumeration of \emph{all balls} centered at points
in $K^{c}$ that contain a given point $x$ gives in particular an
enumeration of the set of all their centers, which is exactly $K^{c}$. 
\end{proof}

\subsection{The strong inclusion approach in metric spaces}

Let $(X,A,\nu,d)$ be a non-necessarily computably separable computable
metric space. Let $\beta$ be the numbering of balls with rational
radii induced by $\nu$. Denote by $\mathring{\subseteq}$ the strong
inclusion of $\beta$ induced by the metric, which comes from the
relation on balls parametrized by pair (point-radius) defined by 
\[
(x,r_{1})\mathring{\subseteq}(y,r_{2})\iff d(x,y)+r_{1}<r_{2}.
\]
We have already defined the Cauchy representation on $X$ and the
representation $\rho_{\beta}^{\mathring{\subseteq}}$ induced by the
numbering of the basis and the strong inclusion relation. As noted
before, we can equivalently replace $<$ by $\le$ in the definition
of $\mathring{\subseteq}$ above.
\begin{prop}
The equivalence $\rho_{Cau}\equiv\rho_{\beta}^{\mathring{\subseteq}}$
holds.
\end{prop}

\begin{proof}
Denote by $t_{n}$ a $c_{\mathbb{Q}}$-name of $2^{-n}$. The map
$n\mapsto t_{n}$ can be supposed computable. If $p$ is a $\rho_{Cau}$-name
for a point $x$, then $q$ defined by 
\[
q(n)=\langle p(n),t_{n}\rangle
\]
defines a $\rho_{\beta}^{\mathring{\subseteq}}$-name of $x$. This
name is given as a computable function of $p$. 

Conversely, suppose that $q$ is a $\rho_{\beta}^{\mathring{\subseteq}}$-name
of $x$. Denote by $\text{fst}$ and $\text{snd}$ the two halves
of the inverse of the pairing function used to define the numbering
$\beta$ of balls in $(X,A,\nu,d)$. Then the following $p$ gives
a $\rho_{Cau}$-name for point $x$:
\[
p(n)=\text{fst}(q(\mu i,\,\text{snd}(q(i))<2^{-n-1}))).
\]
In words: $p(n)$ is defined as the center of the first ball of radius
less than $2^{-n}$ found in the name $q$ of $x$. The fact that
this application of the $\mu$-operator produces a total function
comes exactly from the hypothesis that $\rho_{\beta}^{\mathring{\subseteq}}$-names
give arbitrarily precise information \emph{with respect to the strong
inclusion: }in the $\rho_{\beta}^{\mathring{\subseteq}}$-name of
$x$ appear balls given by arbitrarily small radii.\emph{ }
\end{proof}

\section{\label{sec:Semi-decidable strong inclusion }Semi-decidable strong
inclusion }

In this section we consider semi-decidable strong inclusion relations. 

\subsection{Semi-decidable strong inclusion and c.e. open sets }

One of the purposes of using strong inclusion relations is to be able
to replace set inclusion by a semi-decidable relation. This plays
a crucial role in many applications: \cite{Hertling1996,Spr98,DOWNEY2023}. 

Here we highlight what seems to us to be the most important benefit
of having access to a semi-decidable strong inclusion relation $\mathring{\subseteq}$:
basic open sets are c.e. open with respect to the representation $\rho_{\beta}^{\mathring{\subseteq}}$.
The following is essentially contained in \cite[Lemma 8]{Spreen2001},
we are simply extending it to non-$\text{T}_{0}$ spaces. 
\begin{prop}
Let $(\mathcal{B},\beta)$ be a numbered subbasis for $X$, and let
$\mathring{\subseteq}$ be a strong inclusion for $(\hat{\mathcal{B}},\hat{\beta})$.
Suppose that $\mathring{\subseteq}$ is semi-decidable. Suppose also
that every point of $X$ admits a strong neighborhood basis. 

Then the basic open sets are uniformly c.e. open sets with respect
to the multi-representation $\rho_{\beta}^{\mathring{\subseteq}}$,
i.e.:
\[
\beta\le[\rho_{\beta}^{\mathring{\subseteq}}\rightarrow\rho_{\mathfrak{Si}}].
\]
\end{prop}

\begin{proof}
All we have to show is that there is a Type-2 Turing machine that,
on input the $\rho$-name of a point $x$ and the $\beta$-name $b_{0}$
of a basic set $B$, will stop if and only if $x\in B$. But notice
that, by definition of $\rho_{\beta}^{\mathring{\subseteq}}$, $x\in B$
if and only if any $\rho_{\beta}^{\mathring{\subseteq}}$-name of
$x$ contains a name $b_{1}\in\text{dom}(\beta)$ with $b_{1}\mathring{\subseteq}b_{0}$.
Since we suppose $\mathring{\subseteq}$ semi-decidable, the condition
``containing a name $b_{1}\in\text{dom}(\beta)$ with $b_{1}\mathring{\subseteq}b_{0}$''
is also semi-decidable. 
\end{proof}
We will now show that the result of this proposition is not true in
general: we give an example of a numbered basis $(\mathcal{B},\beta)$,
which defines a representation $\rho_{\beta}^{\min}$, and yet the
elements of $\mathcal{B}$ are not c.e. open for $\rho_{\beta}^{\min}$. 
\begin{exa}
Let $(X,\nu)$ be a numbered set on which equality is not semi-decidable.
(For example: the set of c.e. subsets of $\mathbb{N}$ with numbering
$W$, the computable reals with their numbering.) Consider a numbering
$\beta:\subseteq\mathbb{N}\rightarrow\mathcal{P}(X)$ given by $\text{dom}(\beta)=\text{dom}(\nu)$
and $\forall n\in\text{dom}(\beta),\,\beta(n)=\{\nu(n)\}$. Then we
get a representation $\rho_{\beta}^{\min}$ of $X$. The $\rho_{\beta}^{\min}$-name
of a point $x$ of $X$ is a sequence that contains one or more $\nu$-names
for $x$. But since equality is not semi-decidable for $\nu$, membership
in singletons cannot be semi-decidable. 
\end{exa}

\subsection{Semi-decidability of the strong inclusion and \textquotedblleft totalization
of a numbered basis\textquotedblright . }

In this section, we show that the obvious ``totalization'' of the
numbering of a basis cannot always be achieved while preserving semi-decidability
of a strong inclusion relation. 

Suppose that we start with a partially numbered basis $(\mathcal{B},\beta:\subseteq\mathbb{N}\rightarrow\mathcal{B})$.
We can totalize the numbering $\beta$: we define a numbering $\tilde{\beta}$,
with $\text{dom}(\tilde{\beta})=\mathbb{N}$, by 
\[
\forall n\in\text{dom}(\beta),\,\tilde{\beta}(n)=\beta(n);
\]
\[
\forall n\notin\text{dom}(\beta),\,\tilde{\beta}(n)=\emptyset.
\]

If $(\mathcal{B},\beta)$ was equipped with a strong inclusion relation
$\mathring{\subseteq}$, we can naturally extend it, and define a
strong inclusion $\mathring{\subseteq}'$ for $\tilde{\beta}$ as
follows:
\[
n\mathring{\subseteq}'m\iff(n,m\in\text{dom}(\beta)\,\&\,n\mathring{\subseteq}m)\text{ or }n=m.
\]
Note that if $\mathring{\subseteq}$ was reflexive, $\mathring{\subseteq}'$
remains so. 

There are several other possibilities to define $\mathring{\subseteq}'$,
one could for instance say that $n\mathring{\subseteq}'m$ when $n\notin\text{dom}(\beta)$
and $m\in\text{dom}(\beta)$, this also yields a strong inclusion
relation. 

It is immediate to see that, whatever the chosen extension $\mathring{\subseteq}'$
of $\mathring{\subseteq}$, we get $\rho_{\beta}^{\mathring{\subseteq}}\equiv\rho_{\tilde{\beta}}^{\mathring{\subseteq}'}$.
Indeed, names of the empty set never play a role in the definition
of the representation generated by a basis. 

Consider as a basis for the topology of $\mathbb{R}$ the set $\mathcal{B}$
of open balls with computable reals as center and rational radii.
It is equipped with the numbering $\beta$ given by 
\[
\text{dom}(\beta)=\{\langle n,m\rangle,\,n\in\text{dom}(c_{\mathbb{R}}),\,c_{\mathbb{Q}}(m)>0\},
\]
\[
\forall\langle n,m\rangle\in\text{dom}(\beta),\,\beta(\langle n,m\rangle)=]c_{\mathbb{R}}(n)-c_{\mathbb{Q}}(m),c_{\mathbb{R}}(n)+c_{\mathbb{Q}}(m)[
\]
where $c_{\mathbb{R}}$ is the usual numbering of computable reals
and $c_{\mathbb{Q}}$ is any usual total numbering of $\mathbb{Q}$.
We have the semi-decidable strong inclusion given by 
\[
\langle n_{1},m_{1}\rangle\mathring{\subseteq}\langle n_{2},m_{2}\rangle\iff\vert c_{\mathbb{R}}(n_{1})-c_{\mathbb{R}}(n_{2})\vert+c_{\mathbb{Q}}(m_{1})<c_{\mathbb{Q}}(m_{2}).
\]

\begin{prop}
For any of the totalizations of $\beta$ described above, the resulting
numbered basis cannot be equipped with a semi-decidable strong inclusion
that extends $\mathring{\subseteq}$. 
\end{prop}

\begin{proof}
Consider a $\beta$-name $n$ for $B(0,1)=]-1,1[$. Suppose that $\mathring{\subseteq}'$
is a semi-decidable extension of $\mathring{\subseteq}$ to $\mathbb{N}$.
For $i\in\mathbb{N}$, let $b_{i}$ be the code of a Turing machine
defined as follows: on input $k$, run $\varphi_{i}(k)$, if it stops
output $0$. Thus if $i\in\text{Tot}$, $b_{i}$ is a $c_{\mathbb{R}}$-name
of $0$, otherwise $b_{i}\notin\text{dom}(c_{\mathbb{R}})$. Denote
by $n_{\mathfrak{2}}$ a $c_{\mathbb{Q}}$-name of $2$. Then $\langle b_{i},n_{\mathfrak{2}}\rangle$
is a $\tilde{\beta}$-name of $B(0,2)$ when $i\in\text{Tot}$, and
a name of the empty set otherwise. And then a program semi-deciding
for $\mathring{\subseteq}'$ would semi-decide membership in Tot. 
\end{proof}

\section{\label{sec:Equivalence-of-bases}Equivalence of bases}

In this section, we restrict our attention to numbered bases, instead
of numbered subbases. But subbases are equivalent when the bases they
generate are equivalent, and the same relation holds for notion of
effective equivalence. 

\subsection{Representation \texorpdfstring{$\rho_{\beta}^{\max}$}{ρᵦᵐᵃˣ} and equivalence of bases }

We first note that the representation $\rho_{\beta}^{\max}$ is badly
behaved with respect to equivalence of bases: bases that ``should
be'' equivalent can give non-equivalent representations. We first
show this by an example that uses a non-computably enumerable basis,
and then modify it so that it uses only computably enumerable bases. 

The example used here is that of open balls of $\mathbb{R}$ given
either by rational radii/center or by computable reals for their radii
and center, and a totalized version of this last basis, which fills
every gap with the empty set. In the following section, we present
several notions of equivalence of bases, the first two bases considered
here are equivalent according to all these definitions, and the ``totalized
basis'' is also representation-equivalent and Nogina equivalent to
the other two, but it fails to be Lacombe equivalent to them. 

Denote by $c_{\mathbb{Q}}$ the usual numbering of $\mathbb{Q}$,
which is total. 

Denote by $c_{\mathbb{R}}$ the Cauchy numbering of $\mathbb{R}_{c}$. 

Denote by $\mathfrak{B}_{\mathbb{Q}}$ the set of open intervals of
$\mathbb{R}$ with rational endpoints. 

Define $\beta_{\mathbb{Q}}:\subseteq\mathbb{N}\rightarrow\mathfrak{B}_{\mathbb{Q}}$
by 
\[
\text{dom}(\beta_{\mathbb{Q}})=\{\langle n,m\rangle,\,c_{\mathbb{Q}}(m)>0\};
\]
\[
\beta_{\mathbb{Q}}(\langle n,m\rangle)=B(c_{\mathbb{Q}}(n),c_{\mathbb{Q}}(m)).
\]
The domain of $\beta_{\mathbb{Q}}$ is easily seen to be recursive,
and we can thus in fact suppose that $\beta_{\mathbb{Q}}$ is defined
on all of $\mathbb{N}$. 

Denote by $\mathfrak{B}_{\mathbb{R}}$ the set of open intervals of
$\mathbb{R}$ with computable reals as endpoints. 

Define $\beta_{\mathbb{R}}:\subseteq\mathbb{N}\rightarrow\mathfrak{B}_{\mathbb{R}}$
by 
\[
\text{dom}(\beta_{\mathbb{R}})=\{\langle n,m\rangle,\,c_{\mathbb{R}}(m)>0\};
\]
\[
\beta_{\mathbb{R}}(\langle n,m\rangle)=B(c_{\mathbb{R}}(n),c_{\mathbb{R}}(m)).
\]

Finally, we use a totalization of $\beta_{\mathbb{R}}$, denoted by
$\tilde{\beta}_{\mathbb{R}}$, by adding the empty set to $\mathfrak{B}_{\mathbb{R}}$
and changing the numbering as follows: 
\[
\text{dom}(\tilde{\beta}_{\mathbb{R}})=\mathbb{N};
\]
\[
\forall n\in\text{dom}(\beta_{\mathbb{R}}),\,\tilde{\beta}_{\mathbb{R}}(n)=\beta_{\mathbb{R}}(n),
\]
\[
\forall n\notin\text{dom}(\beta_{\mathbb{R}}),\,\tilde{\beta}_{\mathbb{R}}(n)=\emptyset.
\]

We then have:
\begin{prop}
The representations $\rho_{\beta_{\mathbb{R}}}^{\max}$ and $\rho_{\beta_{\mathbb{Q}}}^{\max}$
are not equivalent. 
\end{prop}

\begin{proof}
This follows from the following strong fact: there is no $\rho_{\beta_{\mathbb{R}}}^{\max}$-computable
point. This follows immediately from the fact that there does not
exist a computable enumeration of all computable reals. 
\end{proof}
\begin{prop}
We have $\rho_{\beta_{\mathbb{R}}}^{\max}\equiv\rho_{\hat{\beta}_{\mathbb{R}}}^{\max}$,
and thus $\rho_{\tilde{\beta}_{\mathbb{R}}}^{\max}$ and $\rho_{\beta_{\mathbb{Q}}}^{\max}$
are not equivalent, even though $\tilde{\beta}_{\mathbb{R}}$ is a
total numbering. 
\end{prop}

\begin{proof}
The identity of Baire space is a realizer for both directions. 
\end{proof}
However, it is very easy to check that for the natural strong inclusion
$\mathring{\subseteq}$ on $\mathbb{R}$, that comes from the metric
of $\mathbb{R}$, we have: 
\begin{prop}
The representations $\rho_{\beta_{\mathbb{R}}}^{\mathring{\subseteq}}$,
$\rho_{\tilde{\beta}_{\mathbb{R}}}^{\mathring{\subseteq}}$ and $\rho_{\beta_{\mathbb{Q}}}^{\mathring{\subseteq}}$
are equivalent. 
\end{prop}

\subsection{Representation-equivalent subbases}

The notion of topological space of Definition \ref{def:Usual Definition }
naturally comes with a notion of equivalence of bases:
\begin{defi}
\label{def Lacombe equivalent}Two numbered bases $(\mathfrak{B}_{1},\beta_{1})$
and $(\mathfrak{B}_{2},\beta_{2})$ are called \emph{Lacombe equivalent}
if there is a program that takes as input the $\beta_{1}$-name of
a basic open set $B_{1}$ and outputs the name of a $\beta_{2}$-computable
sequence $(B_{n})_{n\ge2}$ of basic open sets such that 
\[
B_{1}=\bigcup_{n\ge2}B_{n},
\]
and a program that does the converse operation, with the roles of
$(\mathfrak{B}_{1},\beta_{1})$ and $(\mathfrak{B}_{2},\beta_{2})$
reversed. 
\end{defi}

Recall that associated to a Lacombe basis $(\mathfrak{B}_{1},\beta_{1})$
is a representation of open sets: the name of an open set $O$ is
a sequence $(b_{i})_{i\in\mathbb{N}}\in\text{dom}(\beta_{1})$ such
that 
\[
O=\bigcup_{i\ge0}\beta_{1}(b_{i}).
\]

Definition \ref{def Lacombe equivalent} gives the correct notion
of equivalence of basis with respect to this representation by the
following easy proposition:
\begin{prop}
The numbered bases $(\mathfrak{B}_{1},\beta_{1})$ and $(\mathfrak{B}_{2},\beta_{2})$
induce equivalent representations of open sets if and only if $(\mathfrak{B}_{1},\beta_{1})$
and $(\mathfrak{B}_{2},\beta_{2})$ are Lacombe equivalent.
\end{prop}

Other notions of equivalence of bases can be appropriate, we quote
one to illustrate the variety of possible definitions of equivalence
of bases. The following notion is appropriate to a notion of effective
basis that was first described by Nogina \cite{Nogina1966} for numbered
sets and used in the context of Type 2 computability in \cite{GREGORIADES2016}.
\begin{defi}
\label{def Lacombe equivalent-1}Suppose that $(X,\rho)$ is a second
countable represented space. Two numbered bases $(\mathfrak{B}_{1},\beta_{1})$
and $(\mathfrak{B}_{2},\beta_{2})$ are called \emph{Nogina equivalent}
if there is a program that takes as input the $\beta_{1}$-name of
a basic open set $B_{1}$ and the $\rho$-name of a point $x$ in
$B_{1}$ and outputs the $\beta_{2}$-name of a basic open sets $B_{2}$
such that $x\in B_{2}\subseteq B_{1}$, and a program that does the
converse operation, with the roles of $(\mathfrak{B}_{1},\beta_{1})$
and $(\mathfrak{B}_{2},\beta_{2})$ reversed. 
\end{defi}

In this case, one can also define a certain naturally associated representation
of open sets, and show that two numbered bases define equivalent representations
of open sets exactly when they are Nogina equivalent \cite{Rauzy2023NewDefs}. 

However, neither of these notions of equivalence of bases is appropriate
to the study of the multi-representations $\rho_{\beta_{1}}^{\max}$,
$\rho_{\beta_{1}}^{\min}$ and $\rho_{\beta_{1}}^{\mathring{\subseteq}}$
associated to a numbered subbasis $(B_{1},\beta_{1},\mathring{\subseteq})$.
In particular, Lacombe equivalent bases can yield non-equivalent representations
of points.

We now describe the notion of equivalence of bases appropriate to
the study of the representations $\rho_{\beta_{1}}^{\mathring{\subseteq}}$. 

When $A$ is a set, denote by $A^{*}$ the set of (empty or) finite
sequences of elements of $A$. 

If $\mathring{\subseteq}_{1}$ is a strong inclusion relation on $\text{dom}(\beta_{1})$,
we extend $\mathring{\subseteq}_{1}$ to $\text{dom}(\beta_{1})^{*}$
by 
\[
(b_{1},...,b_{n})\mathring{\subseteq}_{1}(b'_{1},...,b'_{m})\iff\forall i\le m,\exists j\le n,\,b_{j}\mathring{\subseteq}_{1}b'_{i}.
\]

\begin{defi}
\label{def:RPZ equivalence of basis}Consider two numbered bases $(\mathfrak{B}_{1},\beta_{1},\mathring{\subseteq}_{1})$
and $(\mathfrak{B}_{2},\beta_{2},\mathring{\subseteq}_{2})$ of a
set $X$ equipped with strong inclusion relations. We say that $(\mathfrak{B}_{1},\beta_{1},\mathring{\subseteq}_{1})$
is \emph{representation coarser} than $(\mathfrak{B}_{2},\beta_{2},\mathring{\subseteq}_{2})$
if there exists a computable function $f:\subseteq\text{dom}(\beta_{2})^{*}\rightarrow\text{dom}(\beta_{1})^{*}$
defined at least on all sequences $(b_{1},...,b_{n})\in\text{dom}(\beta_{2})^{*}$
such that $\beta_{2}(b_{1})\cap...\cap\beta_{2}(b_{n})\neq\emptyset$,
and such that: 
\begin{itemize}
\item For all sequence $(b_{1},...,b_{n})\in\text{dom}(\beta_{2})^{*}$,
if $f((b_{1},...,b_{n}))=(d_{1},...,d_{m})$, then $\beta_{2}(b_{1})\cap...\cap\beta_{2}(b_{n})\subseteq\beta_{1}(d_{1})\cap...\cap\beta_{1}(d_{m})$\footnote{\noindent In case $f((b_{1},...,b_{n}))$ is just the empty sequence,
we use the convention that an empty intersection\linebreak gives $X$.}. 
\item For all $x$ in $X$ and $d$ in $\text{dom}(\beta_{1})$ with $x\in\beta_{1}(d)$,
for any sequence $(b_{i})_{i\in\mathbb{N}}\in\text{dom}(\beta_{2})$
that defines a strong basis of neighborhood of $x$, there exists
$k\in\mathbb{N}$ such that for all $n\ge k$,
\[
f((b_{1},...,b_{n}))\mathring{\subseteq}_{1}d.
\]
\end{itemize}
We say that $(\mathfrak{B}_{1},\beta_{1},\mathring{\subseteq}_{1})$
and $(\mathfrak{B}_{2},\beta_{2},\mathring{\subseteq}_{2})$ are \emph{representation-equivalent}
if each one is representation coarser than the other. 
\end{defi}

The second condition written above says that as the sequence $(b_{1},b_{2},...)$
closes in on a point, the sequence of images by $f$ should also produce
a strong neighborhood basis of this point. 

We also introduce a more restrictive notion of equivalence of bases,
which implies equivalence of the associated representations, and which
is more natural to work with:
\begin{defi}
\label{def:RPZ equivalence of basis- Uniform }Consider two numbered
bases $(\mathfrak{B}_{1},\beta_{1},\mathring{\subseteq}_{1})$ and
$(\mathfrak{B}_{2},\beta_{2},\mathring{\subseteq}_{2})$ of a set
$X$ equipped with strong inclusion relations. We say that $(\mathfrak{B}_{1},\beta_{1},\mathring{\subseteq}_{1})$
is \emph{uniformly representation coarser} than $(\mathfrak{B}_{2},\beta_{2},\mathring{\subseteq}_{2})$
if there exists a computable  function $f:\subseteq\text{dom}(\beta_{2})^{*}\rightarrow\text{dom}(\beta_{1})^{*}$
defined at least on all sequences $(b_{1},...,b_{n})\in\text{dom}(\beta_{2})^{*}$
such that $\beta_{2}(b_{1})\cap...\cap\beta_{2}(b_{n})\neq\emptyset$,
and such that: 
\begin{itemize}
\item For all sequence $(b_{1},...,b_{n})\in\text{dom}(\beta_{2})^{*}$,
if $f((b_{1},...,b_{n}))=(d_{1},...,d_{m})$, then $\beta_{2}(b_{1})\cap...\cap\beta_{2}(b_{n})\subseteq\beta_{1}(d_{1})\cap...\cap\beta_{1}(d_{m})$. 
\item For all $x$ in $X$ and $d$ in $\text{dom}(\beta_{1})$ with $x\in\beta_{1}(d)$,
there exists $(b_{1},...,b_{n})$ in $\text{dom}(\beta_{2})$, with
$x\in\beta_{2}(b_{1})\cap...\cap\beta_{2}(b_{n})$, and such that
\[
\forall(b'_{1},...,b'_{k})\in\text{dom}(\beta_{2})^{*},\,(b'_{1},...,b'_{k})\mathring{\subseteq}_{2}(b_{1},...,b_{n})\implies f((b'_{1},...,b'_{k}))\mathring{\subseteq}_{1}d.
\]
\end{itemize}
We say that $(\mathfrak{B}_{1},\beta_{1},\mathring{\subseteq}_{1})$
and $(\mathfrak{B}_{2},\beta_{2},\mathring{\subseteq}_{2})$ are \emph{uniformly
representation-equivalent} if each one is uniformly representation
coarser than the other. 
\end{defi}

One checks that the uniform version is more restrictive that the general
notion of representation-equivalence. 
\begin{lem}
Let $(\mathfrak{B}_{1},\beta_{1},\mathring{\subseteq}_{1})$ and $(\mathfrak{B}_{2},\beta_{2},\mathring{\subseteq}_{2})$
be two numbered bases of a set $X$. If $(\mathfrak{B}_{1},\beta_{1},\mathring{\subseteq}_{1})$
is uniformly representation coarser than $(\mathfrak{B}_{2},\beta_{2},\mathring{\subseteq}_{2})$,
then it is also representation coarser than $(\mathfrak{B}_{2},\beta_{2},\mathring{\subseteq}_{2})$.
If $(\mathfrak{B}_{1},\beta_{1},\mathring{\subseteq}_{1})$ and $(\mathfrak{B}_{2},\beta_{2},\mathring{\subseteq}_{2})$
are uniformly representation-equivalent, then they are also representation-equivalent. 
\end{lem}

It is easy to see that $(\mathfrak{B}_{1},\beta_{1},\mathring{\subseteq}_{1})$
being representation coarser than $(\mathfrak{B}_{2},\beta_{2},\mathring{\subseteq}_{2})$
does imply that the topology generated by $\mathfrak{B}_{1}$ as subbasis
is coarser than the topology generated by the subbasis $\mathfrak{B}_{2}$
(in terms of classical mathematics).

The following lemma shows why we have the correct notion of equivalence
of bases.
\begin{lem}
Let $(\mathfrak{B}_{1},\beta_{1},\mathring{\subseteq}_{1})$ and $(\mathfrak{B}_{2},\beta_{2},\mathring{\subseteq}_{2})$
be two numbered bases of a set $X$. Then 
\[
\rho_{\beta_{2}}^{\mathring{\subseteq}_{2}}\le\rho_{\beta_{1}}^{\mathring{\subseteq}_{1}}\iff\text{(\ensuremath{\mathfrak{B}_{1}},\ensuremath{\beta_{1}},\ensuremath{\mathring{\subseteq}_{1}}) is effectively coarser than (\ensuremath{\mathfrak{B}_{2}},\ensuremath{\beta_{2}},\ensuremath{\mathring{\subseteq}_{2}})};
\]
\[
\rho_{\beta_{2}}^{\mathring{\subseteq}_{2}}\equiv\rho_{\beta_{1}}^{\mathring{\subseteq}_{1}}\iff\text{(\ensuremath{\mathfrak{B}_{1}},\ensuremath{\beta_{1}},\ensuremath{\mathring{\subseteq}_{1}}) and (\ensuremath{\mathfrak{B}_{2}},\ensuremath{\beta_{2}},\ensuremath{\mathring{\subseteq}_{2}})}\text{ are representation-equivalent}.
\]
\end{lem}

\begin{proof}
The second equivalence is a direct consequence of the first one, we
thus focus on the first one. 

Suppose first that $(\mathfrak{B}_{1},\beta_{1},\mathring{\subseteq}_{1})$
is effectively coarser than $(\mathfrak{B}_{2},\beta_{2},\mathring{\subseteq}_{2})$,
and thus that we have a function $f$ as in Definition \ref{def:RPZ equivalence of basis}. 

Given the $\rho_{\beta_{2}}^{\mathring{\subseteq}_{2}}$-name of a
point $x$, we show how to compute a $\rho_{\beta_{1}}^{\mathring{\subseteq}_{1}}$-name
of it. 

Simply apply the function $f$ along all initial segments of the $\rho_{\beta_{2}}^{\mathring{\subseteq}_{2}}$-name
of $x$, and output the concatenation of all the results. The fact
that $f$ produces oversets implies that $x$ does belong to all produced
basic open sets. The second condition on $f$ guarantees that we indeed
construct a strong neighborhood basis of $x$. 

Suppose now that $\rho_{\beta_{2}}^{\mathring{\subseteq}_{2}}\le\rho_{\beta_{1}}^{\mathring{\subseteq_{1}}}$. 

By a classical characterization of Type 2 computable functions in
terms of isotone\footnote{A function $f:A^{*}\rightarrow B^{*}$ is called \emph{isotone} if
it is increasing for the prefix relation: if $u$ is a prefix of $v$,
then $f(u)$ is a prefix of $f(v)$. } functions \cite{Schroeder2002}, this implies that there is a computable
isotone function $f:\subseteq\text{dom}(\beta_{2})^{*}\rightarrow\text{dom}(\beta_{1})^{*}$
that testifies for the relation $\rho_{\beta_{2}}^{\mathring{\subseteq}_{2}}\le\rho_{\beta_{1}}^{\mathring{\subseteq_{1}}}$.
This function $f$ is defined at least on all sequences $(b_{1},...,b_{n})\in\text{dom}(\beta_{2})^{*}$
such that $\beta_{2}(b_{1})\cap...\cap\beta_{2}(b_{n})\neq\emptyset$,
because any such sequence is the beginning of the $\rho_{\beta_{2}}^{\mathring{\subseteq}}$-name
of some point. 

Note first that for any sequence $(b_{1},...,b_{n})\in\text{dom}(\beta_{2})^{*}$,
if $f((b_{1},...,b_{n}))=(d_{1},...,d_{m})$, then $\beta_{2}(b_{1})\cap...\cap\beta_{2}(b_{n})\subseteq\beta_{1}(d_{1})\cap...\cap\beta_{1}(d_{m})$.
Indeed suppose that this is not the case. It means that there is a
point $x\in\beta_{2}(b_{1})\cap...\cap\beta_{2}(b_{n})\setminus(\beta_{1}(d_{1})\cap...\cap\beta_{1}(d_{m}))$.
The sequence $(b_{1},...,b_{n})$ could be completed to a $\rho_{\beta_{2}}^{\mathring{\subseteq}}$-name
of $x$, however $f$ cannot map a sequence that starts with $(b_{1},...,b_{n})$
to a $\rho_{\beta_{1}}^{\mathring{\subseteq}}$-name of $x$. This
is a contradiction, and thus the desired inclusion holds. 

Finally, $f$ applied along a strong neighborhood basis (with respect
to $(\mathfrak{B}_{2},\beta_{2},\mathring{\subseteq}_{2})$) of a
point $x$ will always produce a strong neighborhood basis of $x$
(with respect to $(\mathfrak{B}_{1},\beta_{1},\mathring{\subseteq}_{1})$).
This guarantees that the last condition of Definition \ref{def:RPZ equivalence of basis}
is satisfied. 
\end{proof}
The reason why we introduce the uniform version of representation-equivalence
is twofold: 
\begin{itemize}
\item It has a nice interpretation in metric spaces, and it is in general
easier to understand, see the example below.
\item It is unclear whether two bases can be representation-equivalent while
not being uniformly representation-equivalent. We thus ask:
\end{itemize}
\begin{prob}
Can two subbases $(\mathfrak{B}_{1},\beta_{1},\mathring{\subseteq}_{1})$
and $(\mathfrak{B}_{2},\beta_{2},\mathring{\subseteq}_{2})$ of a
set $X$ be representation-equivalent while not being uniformly representation-equivalent?
\end{prob}

\begin{exa}
Suppose we are set in a separable metric space $(X,d)$. There are
many possible numberings of open balls: for any numbering $\nu$ of
a dense subset $A\subseteq X$, and any numbering $c:\subseteq\mathbb{N}\rightarrow T\subseteq\mathbb{R}$
of a set of positive real numbers that has $0$ as an accumulation
point, the numbering 
\[
\beta(\langle n,m\rangle)=B(\nu(n),c(m))
\]
is a numbering of a basis for the topology of $X$. 

For two such numberings $\beta_{1}$ and $\beta_{2}$, the condition
of uniform representation-equivalence with respect to the strong inclusion
of metric spaces says that there is an algorithm that, given a finite
intersection $B_{1}\cap...\cap B_{n}$ of balls given by $\beta_{1}$-names,
covers it by a finite intersection $B'_{1}\cap...\cap B'_{m}\supseteq B_{1}\cap...\cap B_{n}$
of balls given by $\beta_{2}$-names, and, additionally, that when
the minimal radius appearing in the first intersection $B_{1}\cap...\cap B_{n}$
goes to $0$, then the minimal radius appearing in $B'_{1}\cap...\cap B'_{m}$
should go to $0$ as well. 

The condition of representation-equivalence only asks of this algorithm
that along each sequence $(b_{i})_{i\in\mathbb{N}}$ of $\beta_{1}$-names
which defines a sequence of balls with arbitrarily small radii, the
corresponding sequence of $\beta_{2}$-names should also encode small
radii, but the dependence does not have to be uniform in the radii
anymore. 
\end{exa}

\bibliographystyle{alphaurl}
\bibliography{NoteRPZ}

\begin{thebibliography}{GKP16}

\bibitem[Bau00]{Bauer2000}
Andrej Bauer.
\newblock {\em The realizability approach to computable analysis and topology}.
\newblock PhD thesis, School of Computer Science, Carnegie Mellon University, 2000.

\bibitem[BP03]{Brattka2003}
Vasco Brattka and Gero Presser.
\newblock Computability on subsets of metric spaces.
\newblock {\em Theoretical Computer Science}, 305(1-3):43--76, aug 2003.
\newblock \href {https://doi.org/10.1016/s0304-3975(02)00693-x} {\path{doi:10.1016/s0304-3975(02)00693-x}}.

\bibitem[DM23]{DOWNEY2023}
Rodney~G. Downey and Alexander~G. Melnikov.
\newblock Computably compact metric spaces.
\newblock {\em The Bulletin of Symbolic Logic}, 29(2):pp. 170--263, 2023.
\newblock URL: \url{https://www.jstor.org/stable/27226694}.

\bibitem[GKP16]{GREGORIADES2016}
Vassilios Gregoriades, Tamás Kispéter, and Arno Pauly.
\newblock A comparison of concepts from computable analysis and effective descriptive set theory.
\newblock {\em Mathematical Structures in Computer Science}, 27(8):1414--1436, June 2016.
\newblock \href {https://doi.org/10.1017/s0960129516000128} {\path{doi:10.1017/s0960129516000128}}.

\bibitem[Grz55]{Grzegorczyk1955}
Andrzej Grzegorczyk.
\newblock Computable functionals.
\newblock {\em Fundamenta Mathematicae}, 42:168--202, 1955.
\newblock \href {https://doi.org/10.4064/fm-42-1-168-202} {\path{doi:10.4064/fm-42-1-168-202}}.

\bibitem[Her96]{Hertling1996}
Peter Hertling.
\newblock Computable real functions: Type 1 computability versus type 2 computability.
\newblock In Ker{-}I Ko, Norbert~Th. M{\"{u}}ller, and Klaus Weihrauch, editors, {\em Second Workshop on Computability and Complexity in Analysis, {CCA} 1996, August 22-23, 1996, Trier, Germany}, volume {TR} 96-44 of {\em Technical Report}. Unjiversity of Trier, 1996.
\newblock URL: \url{ftp://ftp.informatik.uni-trier.de/pub/Users-Root/reports/96-44/hertling.ps}.

\bibitem[HR16]{Hoyrup2016}
Mathieu Hoyrup and Crist{\'{o}}bal Rojas.
\newblock On the information carried by programs about the objects they compute.
\newblock {\em Theory of Computing Systems}, 61(4):1214--1236, dec 2016.
\newblock \href {https://doi.org/10.1007/s00224-016-9726-9} {\path{doi:10.1007/s00224-016-9726-9}}.

\bibitem[Kle52]{Kleene1952}
Stephen~Cole Kleene.
\newblock {\em Introduction to Metamathematics}.
\newblock Princeton, NJ, USA: North Holland, 1952.

\bibitem[KP22]{Kihara2022}
Takayuki Kihara and Arno Pauly.
\newblock Point degree spectra of represented spaces.
\newblock {\em Forum of Mathematics, Sigma}, 10, 2022.
\newblock \href {https://doi.org/10.1017/fms.2022.7} {\path{doi:10.1017/fms.2022.7}}.

\bibitem[KW85]{Kreitz1985}
Christoph Kreitz and Klaus Weihrauch.
\newblock Theory of representations.
\newblock {\em Theoretical Computer Science}, 38:35--53, 1985.
\newblock \href {https://doi.org/10.1016/0304-3975(85)90208-7} {\path{doi:10.1016/0304-3975(85)90208-7}}.

\bibitem[Lac64]{Lachlan1964}
Alistair~H. Lachlan.
\newblock Effective operations in a general setting.
\newblock {\em Journal of Symbolic Logic}, 29(4):163--178, dec 1964.
\newblock \href {https://doi.org/10.2307/2270370} {\path{doi:10.2307/2270370}}.

\bibitem[Mos64]{Moschovakis1964}
Yiannis Moschovakis.
\newblock Recursive metric spaces.
\newblock {\em Fundamenta Mathematicae}, 55(3):215--238, 1964.
\newblock \href {https://doi.org/10.4064/fm-55-3-215-238} {\path{doi:10.4064/fm-55-3-215-238}}.

\bibitem[Nog66]{Nogina1966}
Elena~Yu. Nogina.
\newblock Effectively topological spaces.
\newblock {\em Doklady Akademii Nauk SSSR}, 169:28--31, 1966.

\bibitem[Pau16]{Pauly2016}
Arno Pauly.
\newblock On the topological aspects of the theory of represented spaces.
\newblock {\em Computability}, 5(2):159--180, may 2016.
\newblock \href {https://doi.org/10.3233/com-150049} {\path{doi:10.3233/com-150049}}.

\bibitem[Rau21]{RauzyV}
Emmanuel Rauzy.
\newblock Computable analysis on the space of marked groups.
\newblock {\em arXiv:2111.01179}, 2021.

\bibitem[Rau23]{Rauzy2023NewDefs}
Emmanuel Rauzy.
\newblock New definitions in the theory of type 1 computable topological spaces.
\newblock {\em arXiv:2311.16340}, 2023.

\bibitem[Sch98]{Schroeder1998}
Matthias Schr{\"o}der.
\newblock Effective metrization of regular spaces.
\newblock In Ker-I Ko, Anil Nerode, Marian~B. Pour-El, Klaus Weihrauch, and Ji\v{r}\'{i} Wiedermann, editors, {\em Computability and Complexity in Analysis}, volume 235 of Informatik Berichte, pages 63--80. FernUniversit{\"a}t Hagen, 1998.

\bibitem[Sch01]{Schroeder2001}
Matthias Schr{\"o}der.
\newblock Admissible representations of limit spaces.
\newblock In {\em Computability and Complexity in Analysis}, pages 273--295. Springer Berlin Heidelberg, 2001.
\newblock \href {https://doi.org/10.1007/3-540-45335-0_16} {\path{doi:10.1007/3-540-45335-0_16}}.

\bibitem[Sch02]{Schroeder2002}
Matthias Schr{\"o}der.
\newblock Extended admissibility.
\newblock {\em Theoretical Computer Science}, 284(2):519--538, jul 2002.
\newblock \href {https://doi.org/10.1016/s0304-3975(01)00109-8} {\path{doi:10.1016/s0304-3975(01)00109-8}}.

\bibitem[Sch03]{Schroeder2003}
Matthias Schr{\"o}der.
\newblock Admissible representations for continuous computations.
\newblock In {\em PhD thesis}, 2003.

\bibitem[Sch21]{Schroeder2021}
Matthias Schr{\"o}der.
\newblock Admissibly represented spaces and qcb-spaces.
\newblock In Vasco Brattka and Peter Hertling, editors, {\em Handbook of Computability and Complexity in Analysis}, pages 305--346. Springer International Publishing, Cham, 2021.
\newblock \href {https://doi.org/10.1007/978-3-030-59234-9_9} {\path{doi:10.1007/978-3-030-59234-9_9}}.

\bibitem[SHT08]{StoltenbergHansen2008}
Viggo Stoltenberg-Hansen and John~V. Tucker.
\newblock Computability on topological spaces via domain representations.
\newblock In {\em New Computational Paradigms}, pages 153--194. Springer New York, 2008.
\newblock \href {https://doi.org/10.1007/978-0-387-68546-5_8} {\path{doi:10.1007/978-0-387-68546-5_8}}.

\bibitem[Sli72]{Slisenko1972}
Anatol~O. Slissenko.
\newblock {\em On constructive nonseparable spaces}, volume 100 of {\em 2}, chapter {i}n Fourteen Papers on Logic, Geometry, Topology and Algebra.
\newblock American Mathematical Society Translations, 1972.
\newblock \href {https://doi.org/10.1090/trans2/029/05} {\path{doi:10.1090/trans2/029/05}}.

\bibitem[Spr98]{Spr98}
Dieter Spreen.
\newblock On effective topological spaces.
\newblock {\em The Journal of Symbolic Logic}, 63(1):185--221, 1998.

\bibitem[Spr01]{Spreen2001}
Dieter Spreen.
\newblock Representations versus numberings: on the relationship of two computability notions.
\newblock {\em Theoretical Computer Science}, 262(1-2):473--499, jul 2001.
\newblock \href {https://doi.org/10.1016/s0304-3975(00)00319-4} {\path{doi:10.1016/s0304-3975(00)00319-4}}.

\bibitem[Tur37]{Turing1937}
Alan~M. Turing.
\newblock On computable numbers, with an application to the {E}ntscheidungsproblem.
\newblock {\em Proceedings of the London Mathematical Society}, s2-42(1):230--265, 1937.
\newblock \href {https://doi.org/10.1112/plms/s2-42.1.230} {\path{doi:10.1112/plms/s2-42.1.230}}.

\bibitem[Wei87]{Weihrauch1987}
Klaus Weihrauch.
\newblock {\em Computability}.
\newblock Springer Berlin Heidelberg, 1987.
\newblock \href {https://doi.org/10.1007/978-3-642-69965-8} {\path{doi:10.1007/978-3-642-69965-8}}.

\bibitem[Wei00]{Weihrauch2000}
Klaus Weihrauch.
\newblock {\em Computable Analysis}.
\newblock Springer Berlin Heidelberg, 2000.
\newblock \href {https://doi.org/10.1007/978-3-642-56999-9} {\path{doi:10.1007/978-3-642-56999-9}}.

\bibitem[Wei03]{Weihrauch2003}
Klaus Weihrauch.
\newblock Computational complexity on computable metric spaces.
\newblock {\em {MLQ}}, 49(1):3--21, jan 2003.
\newblock \href {https://doi.org/10.1002/malq.200310001} {\path{doi:10.1002/malq.200310001}}.

\bibitem[Wei13]{Weihrauch2013}
Klaus Weihrauch.
\newblock Computably regular topological spaces.
\newblock {\em Logical Methods in Computer Science}, Volume 9, Issue 3, August 2013.
\newblock \href {https://doi.org/10.2168/lmcs-9(3:5)2013} {\path{doi:10.2168/lmcs-9(3:5)2013}}.

\bibitem[WG09]{Weihrauch2009ElementaryCT}
Klaus Weihrauch and Tanja Grubba.
\newblock Elementary computable topology.
\newblock {\em J. Univers. Comput. Sci.}, 15:1381--1422, 2009.

\bibitem[WK87]{Weihrauch1987a}
Klaus Weihrauch and Christoph Kreitz.
\newblock Representations of the real numbers and of the open subsets of the set of real numbers.
\newblock {\em Annals of Pure and Applied Logic}, 35:247--260, 1987.
\newblock \href {https://doi.org/10.1016/0168-0072(87)90065-0} {\path{doi:10.1016/0168-0072(87)90065-0}}.

\bibitem[WS81]{Weihrauch1981}
Klaus Weihrauch and Ulrich Schreiber.
\newblock Embedding metric spaces into cpo's.
\newblock {\em Theoretical Computer Science}, 16(1):5--24, 1981.
\newblock \href {https://doi.org/10.1016/0304-3975(81)90027-x} {\path{doi:10.1016/0304-3975(81)90027-x}}.

\end{thebibliography}

\end{document}